\documentclass[11pt]{article}

\usepackage{mathrsfs}
\usepackage[latin1]{inputenc}

\usepackage{amsmath,amsthm,amssymb}
\usepackage{graphicx}
\usepackage{vector}

\font\tencmmib=cmmib10 \skewchar\tencmmib '60
\newfam\cmmibfam
\textfont\cmmibfam=\tencmmib

\def\lessim{\ \lower4pt\hbox{$
\buildrel{\displaystyle <}\over\sim$}\ }
\def\gessim{\ \lower4pt\hbox{$\buildrel{\displaystyle >}
\over\sim$}\ }

\def\vX{\boldsymbol{X}}

\def\la{{\langle}}
\def\ra{{\rangle}}

\newcommand{\e}{\mathbb{E}}
\newcommand{\p}{\mathbb{P}}

\newcommand{\vsi}{\boldsymbol{\sigma}}

\newcommand{\vx}{\boldsymbol{x}}

\newcommand{\bF}{\mathcal{F}}
\newcommand{\bC}{\mathcal{C}}
\newcommand{\bL}{\mathcal{L}}

\newtheorem{lemma}{\bf Lemma}

\newtheorem{theorem}{\bf Theorem}

\newtheorem{remark}{\bf Remark}
\newtheorem{example}{\bf Example}

\newtheorem{proposition}{\bf Proposition}

\newenvironment{Proof of lemma}{\noindent{\bf Proof of Lemma}}{\hfill$\Box$\newline}
\newenvironment{Proof of theorem}{\noindent{\bf Proof of Theorem}}{\hfill{\footnotesize${\square}$}\newline}
\newenvironment{Proof of theorems}{\noindent{\bf Proof of Theorems}}{\hfill$\Box$\newline}
\newenvironment{Proof of proposition}{\noindent{\bf Proof of Proposition}}{\hfill$\Box$\newline}
\newenvironment{Proof of propositions}{\noindent{\bf Proof of Propositions}}{\hfill$\Box$\newline}
\newenvironment{Proof of exercise}{\noindent{\it Proof of Exercise:}}{\hfill$\Box$}

\begin{document}

\nocite{*} 

\title{Variational representations for the Parisi functional and the two-dimensional Guerra-Talagrand bound}

\author{Wei-Kuo Chen \thanks{wkchen@umn.edu} \\ \small{University of Chicago and University of Minnesota} }
\maketitle

\begin{abstract}
The validity of the Parisi formula in the Sherrington-Kirkpatrick model (SK) was initially proved by Talagrand \cite{Tal06}. The central argument therein relied on a very dedicated study of the coupled free energy via the two-dimensional Guerra-Talagrand (GT) replica symmetry breaking bound. It is believed that this bound and its higher dimensional generalization are highly related to the conjectures of temperature chaos and ultrametricity in the SK model, but a complete investigation remains elusive. Motivated by Bovier-Klimovsky \cite{BK} and Auffinger-Chen \cite{AC14}, the aim of this paper is to present a novel approach to analyzing the Parisi functional and the two-dimensional GT bound in the mixed $p$-spin models in terms of optimal stochastic control problems. We compute the directional derivative of the Parisi functional and derive equivalent criteria for the Parisi measure. We demonstrate how our approach provides a simple and efficient control for the GT bound that yields several new results on Talagrand's positivity of the overlap \cite[Section 14.12]{Tal11} and disorder chaos in Chatterjee \cite{Chatt09} and Chen \cite{Chen13}. In particular, we provide some examples of the models containing odd $p$-spin interactions.

\end{abstract}

\section{Introduction}

In 1979, Parisi \cite{P79} suggested an ingenious variational formula for the limiting free energy in the Sherrington-Kirkpatrick (SK) model. Its validity was rigorously established by Talagrand \cite{Tal06} following the beautiful discovery of Guerra's replica symmetry breaking scheme \cite{G03}. In the more general situation, Parisi's formula was later shown to be valid in the mixed $p$-spin models by Panchenko \cite{Pan14}. More precisely, for $N\geq 1,$ the Hamiltonian of the mixed $p$-spin model is defined as
\begin{align}
\label{H}
H_N(\vsi)=X_N(\vsi)+h\sum_{i=1}^N\sigma_i
\end{align}
for $\vsi=(\sigma_1,\ldots,\sigma_N)\in \Sigma_N:=\{-1,+1\}^N$, where $X_N$ is the linear combination of the pure $p$-spin Hamiltonians,
\begin{align*}
X_N(\vsi)=\beta\sum_{p\geq 2}\frac{\gamma_p}{N^{(p-1)/2}}\sum_{1\leq i_1,\ldots,i_p\leq N}g_{i_1,\ldots,i_p}\sigma_{i_1}\cdots\sigma_{i_p}
\end{align*}
for i.i.d. standard Gaussian random variables, $g_{i_1,\ldots,i_p}$, for all $1\leq i_1,\ldots,i_p\leq N$ and $p\geq 2.$ In physics, $g_{i_1,\ldots,i_p}$'s are called the disorder, $h\in\mathbb{R}$ is the strength of the external field and $\beta>0$ is called the (inverse) temperature. Here, we assume that the nonnegative sequence $(\gamma_p)_{p\geq 2}$ decays fast enough, e.g. $\sum_{p\geq 2}2^p\gamma_p^2<\infty$, such that the covariance of $X_N$ can be computed as
\begin{align*}
\e X_N(\vsi^1)X_N(\vsi^2)=N\xi(R_{1,2})
\end{align*}
for any two spin configurations $\vsi^1=(\sigma_1^1,\ldots,\sigma_N^1)$ and $\vsi^2=(\sigma_1^2,\ldots,\sigma_N^2)$ from $\Sigma_N$, where 
\begin{align}\label{overlap}
R_{1,2}:=\frac{1}{N}\sum_{i=1}^N\sigma_i^1\sigma_i^2
\end{align}
is called the overlap between $\vsi^1$ and $\vsi^2$ and 
\begin{align}
\label{eq-4}
\xi(s):=\sum_{p\geq 2}\beta_p^2s^p,\,\,\forall s\in[-1,1]
\end{align}
for $\beta_p:=\beta\gamma_p$ for all $p\geq 2.$
An important example of $\xi$ is the mixed {\it even} $p$-spin model, i.e., $\gamma_p=0$ for all odd $p\geq 3.$ In particular, the SK model corresponds to $\xi(s)=\beta^2s^2/2.$ Denote the Gibbs measure by
\begin{align}\label{Gibbs}
G_N(\vsi)=\frac{\exp H_N(\vsi)}{Z_N},
\end{align}
where the normalizing factor $Z_N=\sum_{\vsi\in\Sigma_N}\exp H_N(\vsi)$ is called the partition function. 

The formulation of the Parisi formula is described as follows. Let $\mathcal{M}$ be the space of all probability measures on $[0,1]$ and $\mathcal{M}_d$ be the collection of all atomic measures in $\mathcal{M}.$ Denote by $\alpha_\mu$ the distribution function of $\mu\in\mathcal{M}.$ We endow the space $\mathcal{M}$ with the metric 
\begin{align}
\label{metric}
d(\mu,\mu')=\int_0^1|\alpha_\mu(s)-\alpha_{\mu'}(s)|ds.
\end{align} 
For any $\mu\in\mathcal{M}$, let $\Phi_{\mu}$ be the solution to the Parisi PDE on $[0,1]\times\mathbb{R}$,
\begin{align}
\begin{split}
\label{PDE}
\partial_s\Phi_{\mu}(s,x)&=-\frac{\xi''(s)}{2}\left(\partial_{xx}\Phi_{\mu}(s,x)+\alpha_\mu(s)(\partial_x\Phi_\mu(s,x))^2\right),\\
\Phi_{\mu}(1,x)&=\log\cosh x.
\end{split}
\end{align} 
Here, for any $\mu\in\mathcal{M}_d$, this PDE can be explicitly solved by performing the Hopf-Cole transformation. As for an arbitrary probability measure $\mu\in\mathcal{M}$, the solution $\Phi_{\mu}$ is understood in the weak sense (see Jagannath and Tobasco \cite{JT}). Define the Parisi functional $\mathcal{P}$ on $\mathcal{M}$ by
\begin{align*}
\mathcal{P}(\mu)&=\log 2+\Phi_{\mu}(0,h)-\frac{1}{2}\int_0^1\alpha_\mu(s)s\xi''(s)ds.
\end{align*}
Note that this functional is Lipschitz continuous (see Guerra \cite{G03}). The famous Parisi formula says that
\begin{align*}
\lim_{N\rightarrow\infty}\frac{1}{N}\e\log Z_N&=\min_{\mu\in\mathcal{M}}\mathcal{P}(\mu).
\end{align*}
Here, the quantity inside the limit of the left-hand side is called the free energy of the model. Recently, it was shown (see Auffinger and Chen \cite{AC14}) that the Parisi functional is strictly convex, which implies the uniqueness of the minimizer. We will call such minimizer the {\it Parisi measure} and denote it by $\mu_P.$ In order to classify the structure of $\mu_P$, we say that the Parisi measure is {\it replica symmetric} (RS) if it is a Dirac measure, is {\it $k$ replica symmetry breaking} ($k$-RSB) if it is atomic with exactly $k+1$ jumps and is {\it full replica symmetry breaking} (FRSB) otherwise. In addition, for  given sequence $(\gamma_p)_{p\geq 2}$ and fixed external field $h$, we define the high temperature regime as the collection of all $\beta>0$ such that the corresponding Parisi measures are RS and the low temperature regime is set as the complement of the former. An important quantity associated to the mixed $p$-spin model is the overlap $R_{1,2}$ between two independently sampled spin configurations $\vsi^1$ and $\vsi^2$ from the Gibbs measure $G_N$. At very high temperature, i.e., when $\beta$ is exceedingly small, this overlap is concentrated around a constant (see Talagrand \cite[Chapter 13]{Tal11} for the SK model and Jagannath and Tobasco \cite{JT2} for the mixed $p$-spin model), whereas in the low temperature regime, it is typically supported by a set containing more than one point (see Panchenko \cite{Pan08}).

Arguably, in the past decade, the most important development in the mean-field spin glasses is Guerra's replica symmetry breaking bound \cite{G03} for the free energy in the mixed even $p$-spin model. Its statement reads that any $N\geq 1$ and $\mu\in\mathcal{M}$, 
\begin{align}\label{RSBI}
\frac{1}{N}\e\log Z_N\leq \mathcal{P}(\mu).
\end{align}
Based on Guerra's interpolation scheme \cite{G03}, the first proof of Parisi's formula was obtained in the seminal work of Talagrand \cite{Tal06}, where the central ingredient was played by a two-dimensional extension of Guerra's inequality \eqref{RSBI} for the coupled free energy with constrained overlaps, which was used to control the error estimate between the two sides of \eqref{RSBI} when $\mu$ is very close to the Parisi measure. Later the fully generalization of Guerra's inequality \eqref{RSBI}, called the Guerra-Talagrand (GT) bound throughout this paper, was presented in \cite[Section 15.7]{Tal11}. The two-dimensional GT bound, in particular, has two important consequences regarding the behavior of the overlap under the Gibbs measure. 
The first is known as the positivity of the overlap established by Talagrand \cite[Section 14.12]{Tal11} in the mixed even $p$-spin model, which says that if the external field is present, $h\neq 0,$ then the overlap defined above is essentially bounded from below by some positive constant. Note that this behavior is very different from the one when the external field is absent, $h=0,$ in which case the overlap $R_{1,2}$ is symmetric with respect to the origin.

Another consequence is concerned with the phenomenon of chaos in disorder. It arose from the observation that in some spin glass models, a small perturbation to the disorder will result in a dramatic change to the overall energy landscape (see Rizzo \cite{R2009} for a recent survey in physics). In the mixed $p$-spin model, one typical way to measure such instability  is to consider two Hamiltonians,
\begin{align*}
H_N^1(\vsi^1)&=X_N^1(\vsi^1)+h\sum_{1\leq i\leq N}\sigma_i^1\,\,\mbox{and}\,\,H_N^2(\vsi^2)=X_N^2(\vsi^2)+h\sum_{1\leq i\leq N}\sigma_i^2,
\end{align*} 
where $X_N^1$ and $X_N^2$ are jointly Gaussian with mean zero and covariance structure, 
\begin{align}
\begin{split}\label{eq--2}
\e X_N^1(\vsi^1)X_N^1(\vsi^2)&=\xi(R_{1,2})=\e X_N^2(\vsi^1)X_N^2(\vsi^2),\\
\e X_N^1(\vsi^1)X_N^2(\vsi^2)&=t\xi(R_{1,2})
\end{split}
\end{align}
for some $t\in[0,1].$ Let $\vsi^1$ and $\vsi^2$ be independent samplings from $G_N^1$ and $G_N^2$ respectively and let $R_{1,2}$ be their overlap, which now also depends on $t.$ The case $t=1$ means that the two systems are the same, $H_N^1=H_N^2=H_N$, and the overlap has the behavior we described before. From physics literature (e.g. Bray-Moore \cite{BM87}, Fisher-Huse \cite{FH86}, Kr\c{z}aka\l{}a-Bouchaud \cite{KB05}), chaos in disorder is defined by the phenomenon that $R_{1,2}$ is concentrated around a nonrandom number independent of $N$ if the two systems are decoupled, i.e., $t\in (0,1).$ The key point here is that such behavior is predicted to be true at any temperature. The first rigorous result along this direction was justified in the mixed even $p$-spin models without external field in the work of Chatterjee \cite{Chatt09} and the situation in the presence of the external field was carried out in Chen \cite{Chen13}. 

As the above discussion indicates, the Parisi functional and the GT bound have played fundamental roles in the study of the mixed $p$-spin model. Several challenging conjectures, such as the strong ultrametricity and temperature chaos (see Talagrand \cite[Section 15.7]{Tal11}), rely heavily on the subtle control of these two objects and their higher dimensional generalization. To this regard, the aim of this paper is to present a novel approach to analyzing the Parisi functional as well as the two-dimensional GT bound by means of the optimal stochastic control theory. Ultimately, we hope that this new method will shed some light on how to tackle the remaining open problems. Our idea is motivated by the observation that the Parisi PDE solution $\Phi_\mu$ admits a variational representation (see Theorem \ref{thm-2} below) in terms of an optimal stochastic control problem that corresponds to the Hamilton-Jacobi-Bellman equation induced by a linear diffusion control problem. This was formerly used in Bovier and Klimovsky \cite{BK} to study the strict convexity of the Parisi functional for some cases of the SK model with multidimensional spins. Later it was understood that this approach allows to derive the strict convexity of the Parisi functional in the mixed $p$-spin models by Auffinger and Chen \cite{AC14}.

This article consists of four major results. The first part gives an analytic study of the Parisi formula, where we compute the directional derivative of the Parisi functional and give equivalent criteria for the Parisi measure. As an application, we generalize a theorem of Toninelli \cite{Ton02}, which states that the Parisi measure in the SK model is not a Dirac measure when the temperature and external field stay above the Almeida-Thouless transition line (see \eqref{ex1:eq2} below). In addition, we extend Talagrand's characterization \cite[Theorem 13.4.1]{Tal11} of the high temperature regime for the SK model to the temperature regime of $k$-RSB Parisi measures for any mixed $p$-spin models. Second, we establish a variational representation for the two-dimensional Parisi PDE solution in terms of an optimal stochastic control problem and use this to give a new formulation of the original GT bound. Based on this new form, our last two results are devoted to demonstrating a self-contained proof to establish the positivity of the overlap and disorder chaos in the mixed $p$-spin model. We recover the aforementioned results and furthermore, extend them to many new examples of the model allowing odd $p$-spin interactions. Along the way, we also obtain a nonnegativity principle of the overlap in the mixed $p$-spin model, which says that in the absence of the external field, the overlap is basically nonnegative if one adds certain odd $p$-spin interactions to the Hamiltonian. As one shall see in Section \ref{sec5} below, our approach significantly simplifies and avoids several technicalities in the control of the two-dimensional GT bound compared to the arguments in Talagrand \cite[Section 14.12]{Tal11} and Chen \cite{Chen13}. For instance,  the error estimate of this bound was previously obtained through a quite involved iteration for certain functions of Gaussian random variables. With the new approach, it now becomes quantitatively simpler in the critical case (see Proposition \ref{lem6} below). 

This paper is organized as follows. In Section 2, we state the four main results described above and their proofs are presented in the following three sections. The analytic properties of the Parisi functional is investigated in Section 3 and the variational representation for the two-dimensional GT bound is derived in Section 4. Finally, we present the proof for the results on the positivity of the overlap and disorder chaos in Section 5.

\smallskip
\smallskip

{\noindent \bf Acknowledgements.} The author thanks Arnab Sen 
for several suggestions regarding the presentation of the paper. This research is partially supported by the AMS-Simons Travel Grant.

\section{Main results}

\subsection{Some properties of Parisi's functional and measure}\label{sub1.1}

First, we recall the variational representation for the Parisi PDE from Auffinger and Chen \cite{AC14}. Let $(\p,\mathscr{F},(\mathscr{F}_r)_{0\leq r\leq 1})$ be a filtrated probability space satisfying the usual condition, i.e., it is complete and the filtration is right continuous. Let $B=\{B(r),\mathscr{F}_r,0\leq r\leq 1\}$ be a standard Brownian motion. For $0\leq s<t\leq 1,$ let $D[s,t]$ be the collection of all progressively measurable processes $u$ with respect to $(\mathscr{F}_r)_{s\leq r\leq t}$ satisfying $\sup_{s\leq r\leq t}|u(r)|\leq 1$. We equip the space $D[s,t]$ with the norm $\|u\|=(\int_s^t\e u(w)^2dw)^{1/2}$. Let $\xi$ and $h$ be fixed. Set $\zeta=\xi''.$ For $\mu\in\mathcal{M}$, we define a functional
\begin{align}\label{eq2}
F_\mu^{s,t}(u,x)&=\e [C_\mu^{s,t}(u,x)-L_\mu^{s,t}(u)]
\end{align}
for $u\in D[s,t]$ and $x\in\mathbb{R},$ where letting $\alpha_\mu$ be the distribution function of $\mu,$
\begin{align*}
C_\mu^{s,t}(u,x)&:=\Phi_\mu\left(t,h+\int_s^t\alpha_\mu(w)\zeta(w)u(w)dw+\int_s^t\zeta(w)^{1/2}dB(w)\right),\\
L_\mu^{s,t}(u)&:=\frac{1}{2}\int_s^t \alpha_\mu(w)\zeta(w)u(w)^2dw.
\end{align*}
The Parisi PDE solution can be expressed as

\begin{theorem}[{$\cite[\mbox{Theorem 3 and Proposition 3}]{AC14}$}]\label{thm-2} We have
	\begin{align}\label{thm-2:eq1}
	\Phi_\mu(s,x)&=\max_{u\in D[s,t]}F_\mu^{s,t}(u,x).
	\end{align}
	Here, the maximum is attained by $u_\mu(r)=\partial_x\Phi_\mu(r,X(r))$, where $(X(r))_{s\leq r\leq t}$ satisfies
	\begin{align}\label{max}
	X(r)&=x+\int_s^r\alpha_\mu(w)\zeta(w)\partial_x\Phi_\mu(w,X(w))dw+\int_s^r\zeta(w)^{1/2}dB(w).
	\end{align}
	In addition, the maximizer is unique if $\alpha_\mu>0$ on $[s,t]$ and $\int_s^t\alpha_\mu(r)dr<1.$
\end{theorem}

Here and thereafter, the existence of the partial derivatives of $\partial_{x}\Phi_{\mu}$ and $\partial_{xx}\Phi_{\mu}$ is ensured by \cite[Proposition 2]{AC13}.  
Letting $(s,t)=(0,1)$ in Theorem \ref{thm-2}, the Parisi functional now reads
\begin{align*}
\mathcal{P}(\mu)&=\log 2+\max_{u\in D[0,1]}\left(F_\mu^{0,1}(u,h)-\frac{1}{2}\int_0^1\alpha_\mu(w)w\zeta(w)dw\right).
\end{align*}
Our first main results below are the computation of the directional derivative of the Parisi functional and  the equivalent criteria for the Parisi measure.

\begin{theorem}\label{prop1}
	Let $\mu_0\in\mathcal{M}.$ Define $\mu_\theta=(1-\theta)\mu_0+\theta\mu$ for each $\mu\in\mathcal{M}$ and $\theta\in[0,1].$  We have
	\begin{align}\label{prop1:eq1}
	\left.\frac{d}{d\theta}\mathcal{P}(\mu_\theta)\right|_{\theta=0}=
	\frac{1}{2}\int_0^1\zeta(r)(\alpha_\mu(r)-\alpha_{\mu_0}(r))(\e u_{\mu_0}(r)^2-r)dr
	\end{align} 
	for all $\mu\in\mathcal{M},$ where $\left.\frac{d}{d\theta}\mathcal{P}(\mu_\theta)\right|_{\theta=0}$ is understood as the right derivative at $0$ and $u_{\mu_0}$ is the maximizer of \eqref{thm-2:eq1} using $\mu_0$ and $(s,t)=(0,1)$. In addition, the following statements are equivalent
	\begin{itemize}
		\item[$(i)$] $\mu_0$ is the Parisi measure.
		\item[$(ii)$] $\left.\frac{d}{d\theta}\mathcal{P}(\mu_\theta)\right|_{\theta=0}\geq 0$ for all $\mu\in\mathcal{M}.$
		\item[$(iii)$] $\left.\frac{d}{d\theta}\mathcal{P}(\mu_\theta)\right|_{\theta=0}\geq 0$ for all Dirac measures $\mu=\delta_q$ with $q\in[0,1].$
	\end{itemize}
\end{theorem}

\smallskip

The equivalence of $(i)$ and $(ii)$ is mainly due to the strict convexity of the Parisi functional. The criterion $(iii)$ essentially says that if one could not lower the Parisi functional by adding one more jump to $\mu_0$, then $\mu_0$ must be the Parisi measure. There are two immediate consequences that can be drawn from this theorem. For convenience, we set $\mathcal{M}_d^k$ for $k\geq 0$ to be the collection of all members in $\mathcal{M}_d$ that have no more than $k+1$ atoms. In particular, $\mathcal{M}_d^0$ denotes the space of all Dirac measures on $[0,1]$. In the first consequence, we extract some information about the support of the Parisi measure.

\begin{proposition}\label{thm1}
	Let $S$ be the support of $\mu_P$. For all $q\in S,$
	\begin{align}
	\begin{split}
	\label{thm1:eq1}
	\e\partial_x\Phi_{\mu_P}(q,X(q))^2&=q,
	\end{split}\\
	\begin{split}\label{thm1:eq2}
	\zeta(q)\e \partial_{xx}\Phi_{\mu_P}(q,X(q))^2&\leq 1,
	\end{split}
	\end{align} 
	where $(X(s))_{0\leq s\leq 1}$ satisfies the following stochastic differential equation,
	\begin{align*}
	X(s)&=h+\int_0^s\alpha_{\mu_P}(r)\zeta(r)\partial_x\Phi_{\mu_P}(w,X(w))dw+\int_0^s\zeta(w)^{1/2}dB(w),\,\,\forall s\in[0,1].
	\end{align*}
\end{proposition}

\begin{remark}
	\rm Suppose that $\mu_P$ is a Dirac measure at some $q\in [0,1]$. A direct computation gives 
	\begin{align*}
	\Phi_{\mu_P}(s,x)&=
	\left\{
	\begin{array}{ll}
	\frac{1}{2}(\xi'(1)-\xi'(s))+\e\log\cosh\left(x+z(\xi'(q)-\xi'(s))^{1/2}\right),&\mbox{if $(s,x)\in[0,q)\times\mathbb{R}$,}\\
	\frac{1}{2}(\xi'(1)-\xi'(q))+\log\cosh x,&\mbox{if $(s,x)\in[q,1]\times\mathbb{R}$},
	\end{array}
	\right.
	\end{align*}
	for some standard Gaussian random variable $z.$ Since $\alpha_{\mu_P}=0$ on $[0,q)$, Theorem \ref{thm1} reads
	\begin{align}\label{ex1:eq2}
	\e\tanh^2\left(z\xi'(q)^{1/2}+h\right)&=q,\notag\\
	\zeta(q)\e\frac{1}{\cosh^4\left(z\xi'(q)^{1/2}+h\right)}&\leq 1.
	\end{align}
	Note that if $q\in[0,1]$ minimizes the Parisi functional over all choices in $\mathcal{M}_d^0$, then one can get the first equation (by a direct differentiation, see e.g. \cite[Chapter 1]{Tal10}). But if the temperature and external field are above the Almeida-Thouless line, i.e., \eqref{ex1:eq2} is violated, then the Parisi measure can not be RS. This generalizes Toninelli's theorem \cite{Ton02}, where he established the same statement for the SK model $\xi(s)=\beta^2s^2/2.$ 
\end{remark}

\begin{remark}\rm Consider the SK model without external field, i.e., $\xi(s)=\beta^2s^2/2$ and $h=0.$ We now argue that the high temperature regime, defined as the collection of all $\beta$ such that $\mu_P$ is a Dirac measure, is described by $\beta\leq 1$. To see this, note that since $h=0,$ $0$ is always in the support of the Parisi measure by  \cite[Theorem 1]{AC13}. Thus, it suffices to show that $\mu_P=\delta_0$ if and only if $\beta\leq 1.$ If $\mu_P=\delta_0$ and $\beta>1,$ we will obtain a contradiction as \eqref{ex1:eq2} is violated. Conversely, suppose $\beta\leq 1$. A use of It\^o's formula and \eqref{PDE} gives
	\begin{align*}
	u_{\delta_0}(r)&=\beta\int_0^r\partial_{xx}\Phi_{\delta_0}(w,X(w))dB(w)+u_{\delta_0}(0)=\beta\int_0^r\frac{1}{\cosh^2 X(w)}dB(w)
	\end{align*}
	and hence,
	\begin{align*}
	\e u_{\delta_0}(r)^2&=\beta^2\int_0^r\frac{1}{\cosh^4 X(w)}dw\leq \beta^2\int_0^r1dw\leq r.
	\end{align*}
	Therefore, for all $\mu\in\mathcal{M},$
	\begin{align*}
	\frac{d}{d\theta}\mathcal{P}(\mu_\theta)\Bigr|_{\theta=0}&=\frac{\beta^2}{2}\int_0^1(\alpha_\mu(r)-1)(\e u_{\delta_0}(r)^2-r)dr\geq 0
	\end{align*}
	and Theorem \ref{prop1} implies that $\delta_0$ is the Parisi measure. 
\end{remark}

The second consequence of Theorem \ref{prop1} is a generalization of Talagrand's characterization \cite[Theorem 13.4.1]{Tal11} of the high temperature regime for the SK model, $\xi(s)=\beta^2s^2/2$, where he showed that this regime is indeed equal to the set of all $\beta$ such that $\inf_{\mu\in \mathcal{M}_d^1}\mathcal{P}(\mu)=\mathcal{P}(\mu_0)$ for some $\mu_0\in\mathcal{M}_d^0.$ For any such $\beta$, he proved that $\mu_0$ will automatically be the Parisi measure. With the help of Theorem \ref{prop1} $(iii)$, this result can be generalized to any $k$-RSB Parisi measures.

\begin{proposition}\label{prop4} 
	Consider arbitrary $\xi$ and $h$. Let $k\geq 0$ and $\mu_0$ be an optimizer of $\mathcal{P}$ over $\mathcal{M}_d^k.$ If
	\begin{align}
	\label{prop4:eq1}
	\inf_{\mu\in\mathcal{M}_d^{k+1}}\mathcal{P}(\mu)=\mathcal{P}(\mu_0),
	\end{align}
	then $\mu_0$ is the Parisi measure. 
\end{proposition}

In other words, for fixed sequence $(\gamma_p)_{p\geq 2}$ and external field $h$, the temperature regime of $k$-RSB Parisi measures is described by the collection of all $\beta>0$ such that the corresponding Parisi functionals satisfy \eqref{prop4:eq1} for some optimizer $\mu_0$ of $\mathcal{P}$ restricted to $\mathcal{M}_d^k.$ 

It is generally very difficult to compute the Parisi measure as one needs to minimize $\mathcal{P}$ over all probability measures on $[0,1].$ In principle, Proposition \ref{prop4}  suggests a heuristic way to simulate $k$-RSB Parisi measures. The procedure is based on the observation that if we restrict $\mathcal{P}$ to $\mathcal{M}_d^k$, then it is a differentiable function that depends only on $2(k+1)$ variables on a compact set, $$
\Bigl\{(q_1,\ldots,q_{k+1},a_1,\ldots,a_{k+1}):0\leq q_1\leq \cdots\leq q_{k+1}\leq 1,0\leq a_1,\cdots,a_{k+1}\leq 1,\sum_{i=1}^{k+1}a_i=1\Bigr\},$$ 
on which one can compute the derivative of $\mathcal{P}$ and numerically simulate the minimizer of $\mathcal{P}$ over $\mathcal{M}_d^k.$ Starting from the case $k=0$, if \eqref{prop4:eq1} is satisfied, then one can stop and obtain the RS Parisi measure; otherwise one must proceed to the case $k=1$ and continue this process. If eventually there is a smallest integer $k\geq 0$ such that \eqref{prop4:eq1} is obtained, then one gets a $k$-RSB Parisi measure. 

\subsection{A variational representation for the two-dimensional GT bound}\label{sec2}
The two-dimensional GT bound in the setting of \cite[Theorem 15.7]{Tal11} is formulated as follows. Let $h_1,h_2\in\mathbb{R}$ and $X_N^1,X_N^2$ be jointly Gaussian processes indexed by $\Sigma_N$ with mean zero and covariance,
\begin{align*}
\e X_{N}^\ell(\vsi^1)X_N^{\ell'}(\vsi^2)=N\xi_{\ell,\ell'}(R_{1,2})
\end{align*} 
for $1\leq \ell,\ell'\leq 2$ and $\vsi^1,\vsi^2\in\Sigma_N$, where $R_{1,2}$ is the overlap between $\vsi^1,\vsi^2$ defined through \eqref{overlap}. Here $\xi_{\ell,\ell'}$'s are convex functions on $[-1,1]$ defined in terms of infinite series as $\xi$ in \eqref{eq-4}. Consider two mixed $p$-spin Hamiltonians,
\begin{align}\label{eq0}
H_N^\ell(\vsi^\ell)&=X_N^\ell(\vsi^\ell)+h_\ell\sum_{1\leq i\leq N}\sigma_i^\ell,\,\,\ell=1,2.
\end{align}
 Denote by $S_N$ the collection of all possible values that $R_{1,2}$ could attained. Fix $q\in S_N.$ Assume that $(y_p^\ell)_{0\leq p\leq k}$ for $1\leq \ell\leq 2$ are jointly centered Gaussian random variables such that for certain real sequences $(\rho_{p}^{\ell,\ell'})_{0\leq p\leq k+1}$ for $1\leq \ell,\ell'\leq 2$ with
\begin{align}
\label{eq-2}
\rho_{0}^{1,1}=\rho_0^{2,2}=\rho_0^{1,2}=\rho_0^{2,1}=0,\,\,\rho_{k+1}^{1,1}=\rho_{k+1}^{2,2}=1,\,\,\rho_{k+1}^{1,2}=\rho_{k+1}^{2,1}=q,
\end{align}
we have 
$$
\e y_p^{\ell}y_p^{\ell'}=\xi_{\ell,\ell'}'(\rho_{p+1}^{\ell,\ell'})-\xi_{\ell,\ell'}'(\rho_p^{\ell,\ell'}).
$$

\begin{theorem}[Guerra-Talagrand]\label{thm0}
Let $(m_p)_{0\leq p\leq k}$ be a sequence with $m_0=0<m_1<\cdots<m_{k-1}<m_k=1.$ Under the assumptions stated above, we have that
\begin{align}
\begin{split}
\label{thm0:eq1}
F_{N}(q)&:=\frac{1}{N}\e\log\sum_{R_{1,2}=q}\exp\bigl(H_N^1(\vsi^1)+H_N^2(\vsi^2)\bigr)\\
&\leq 2\log 2+Y_0-\lambda q-\frac{1}{2}\sum_{1\leq \ell,\ell'\leq 2}\sum_{p=0}^{k}m_p(\theta_{\ell,\ell'}(\rho_{p+1}^{\ell,\ell'})-\theta_{\ell,\ell'}(\rho_{p}^{\ell,\ell'})),
\end{split}
\end{align}
where $\theta_{\ell,\ell'}(s):=s\xi_{\ell,\ell'}'(s)-\xi_{\ell,\ell'}(s)$ and $Y_0$ is defined as follows. Denote by $\e_p$ the expectation with respect to $y_p^{\ell,\ell'}.$ Starting with 
\begin{align*}
Y_{k+1}&=\log\Bigl(\cosh\Bigl(h_1+\sum_{p=0}^ky_p^1\Bigr)\cosh\Bigl(h_2+\sum_{p=0}^ky_p^2\Bigr)\cosh\lambda\\
&\qquad\quad+\sinh\Bigl(h_1+\sum_{p=0}^ky_p^1\Bigr)\sinh\Bigl(h_2+\sum_{p=0}^ky_p^2\Bigr)\sinh\lambda\Bigr),
\end{align*}
we define decreasingly $Y_p=m_p^{-1}\e_p\exp m_p Y_{p+1}$ for $1\leq p\leq k$. Finally, set $Y_0=\e_0 Y_1.$
\end{theorem}

The inequality \eqref{thm0:eq1} is a two-dimensional extension of Guerra's replica symmetry breaking bound \eqref{RSBI}. 
Its proof as well as the higher dimensional extension can be found in \cite[Section 15.7]{Tal11}. Recall $q$ from the statement of Theorem \ref{thm0}. Let $\iota=1$ if $q\geq 0$ and $\iota=-1$ otherwise. For $1\leq \ell,\ell'\leq 2$, let $\rho_{\ell,\ell'}$  be nondecreasing continuous functions on $[0,1]$ with
\begin{align}
\begin{split}
\label{eq--1}
&\rho_{1,1}(0)=\rho_{1,2}(0)=\rho_{2,1}(0)=\rho_{2,2}(0)=0,\\
&\rho_{1,1}(1)=\rho_{2,2}(1)=1,\,\,\rho_{1,2}(1)=\rho_{2,1}(1)=|q|.
\end{split}
\end{align}
Assume that these functions are differentiable everywhere except  at a finite number of points, on which the right derivatives exist.
For any $s\in [0,1]$, we define 
\begin{align}\label{eq-10}
T(s)&=\left[
\begin{array}{cc}
\zeta_{1,1}(s)&\zeta_{1,2}(s)\\
\zeta_{2,1}(s)&\zeta_{2,2}(s)
\end{array}\right]:=
\left[
\begin{array}{cc}
\frac{d}{ds}\xi_{1,1}'(\rho_{1,1}(s))&\frac{d}{ds}\xi_{1,2}'(\iota\rho_{1,2}(s))\\
\frac{d}{ds}\xi_{2,1}'(\iota \rho_{2,1}(s))&\frac{d}{ds}\xi_{2,2}'(\rho_{2,2}(s))
\end{array}\right]
\end{align}
In the right-hand side of \eqref{eq-10}, the derivatives are understood as the ones from the right if one of $\rho_{\ell,\ell'}$'s is not differentiable. We suppose that $T(s)$ is positive semi-definite and its operator norm $\|T(s)\|$  is uniformly bounded from above by some constant $K>0.$ For $\mu\in\mathcal{M}_d$, we consider the classical solution $\Psi_{\mu}$ to the two-dimensional Parisi PDE,
\begin{align}
\begin{split}
\label{eq-9}
\partial_s\Psi_{\mu}&=-\frac{1}{2}\left(\left<T,\triangledown^2\Psi_{\mu}\right>+\alpha_\mu\left<T\triangledown \Psi_{\mu},\triangledown\Psi_{\mu}\right>\right)
\end{split}
\end{align}
for $(\lambda,s,\vx)\in\mathbb{R}\times [0,1)\times\mathbb{R}^2$ with terminal condition
\begin{align}
\begin{split}
\label{tc}
\Psi_{\mu}(\lambda,1,\vx )&=\log\left(\cosh x_1\cosh x_2\cosh \lambda+\sinh x_1\sinh x_2 \sinh\lambda\right).
\end{split}
\end{align}
The assumption $\mu\in\mathcal{M}_d$ guarantees the existence of the solution by a usual application of Hopf-Cole transformation. One may refer to Lemma \ref{lem3} below for the precise formula of the solution. Our first main result below says that the mapping $\mu\in \mathcal{M}_d\mapsto\Psi_{\mu}$ is Lipschitz with respect to the metric $d$ defined by \eqref{metric}.

\begin{theorem} \label{sec1.2:thm1}
For any $\mu,\mu'\in\mathcal{M}_d$, we have that
\begin{align*}
|\Psi_{\mu}(\lambda,s,\vx)-\Psi_{\mu'}(\lambda,s,\vx)|&\leq 3Kd(\mu,\mu')
\end{align*}
for $(\lambda,s,\vx)\in\mathbb{R}\times[0,1]\times\mathbb{R}^2$. 
\end{theorem}

This Lipschitz property allows us to extend $\Psi_\mu$ continuously to all $\mu\in\mathcal{M}$ by using sequences of atomic probability measures. Denote by $\mathcal{B}=\{\mathcal{B}(r)=(\mathcal{B}_1(r),\mathcal{B}_2(r)),\mathscr{G}_r,0\leq r<\infty\}$ a two-dimensional Brownian motion, where $(\mathscr{G}_r)_{r\geq 0}$ satisfies the usual condition. For $0\leq s< t\leq 1,$ denote by $\mathcal{D}[s,t]$ the space of all two-dimensional progressively measurable processes $v=(v_1,v_2)$ with respect to $(\mathscr{G}_r)_{s\leq r\leq t}$ satisfying $
\sup_{s\leq r\leq t}|v_1(r)|\leq 1$ and $\sup_{s\leq r\leq t}|v_2(r)|\leq 1.$ Endow the space $\mathcal{D}[s,t]$ with the norm 
\begin{align*}
\|v\|_{s,t}&=\Bigl(\e\int_s^t(v_1(w)^2+v_2(w)^2)dw\Bigr)^{1/2}.
\end{align*}
Similar to the formulation of \eqref{eq2}, we define a functional
\begin{align*}
\bF_\mu^{s,t}(\lambda,v,\vx)&=\e\left[\bC_\mu^{s,t}(\lambda,v,\vx)-\bL_\mu^{s,t}(v)\right]
\end{align*}
for  $(\lambda,v,\vx)\in\mathbb{R}\times\mathcal{D}[s,t]\times\mathbb{R}^2,$ where
\begin{align*}
\bC_\mu^{s,t}(\lambda,v,\vx)&:=\Psi_\mu\Bigl(\lambda,t,\vx+\int_s^t\alpha_\mu(w)T(w)v(w)dw+\int_s^tT(w)^{1/2}d\mathcal{B}(w)\Bigr),\\
\bL_\mu^{s,t}(v)&:=\frac{1}{2}\int_s^t\alpha_\mu(w)\left<T(w)v(w),v(w)\right>dw.
\end{align*}
The following is an analogue of Theorem \ref{thm1} for $\Psi_{\mu}$.

\begin{theorem}\label{thm2}  We have
\begin{align}
\label{thm2:eq2}
\Psi_\mu(\lambda,s,\vx)&=\max\left\{\bF_\mu^{s,t}(\lambda,v,\vx)\big|v\in \mathcal{D}[s,t]\right\}.
\end{align} 
Here the maximum of \eqref{thm2:eq2} is attained by $v_{\mu}(r)=\triangledown \Psi_\mu(\lambda,r,\vX(r))$, where the two-dimensional stochastic process $(\vX(r))_{s\leq r\leq t}$ satisfies
\begin{align}
\label{thm2:eq1}
\vX(r)&=x+\int_s^r\alpha_\mu(w)T(w)\triangledown \Psi_\mu(\lambda,w,\vX(w))dw+\int_s^rT(w)^{1/2}d\mathcal{B}(w).
\end{align}
\end{theorem}

Using the notations introduced above, we can now formulate the GT bound in terms of $\Psi_\mu.$

\begin{theorem}[Guerra-Talagrand]\label{thm-1}
Suppose that $T$ is positive semi-definite for all $s$. Then
\begin{align}
\begin{split}\label{thm-1:eq1}
F_N(q)&\leq 2\log 2+\e\Psi_{\mu}(\lambda,0,h_1,h_2)-\lambda q\\
&\quad-\frac{1}{2}\int_0^1\alpha_\mu(s)\Biggl(\sum_{\ell=\ell'}\rho_{\ell,\ell'}(s)\zeta_{\ell,\ell'}(s)+\iota\sum_{\ell\neq\ell'}\rho_{\ell,\ell'}(s)\zeta_{\ell,\ell'}(s)\Biggr)ds.
\end{split}
\end{align}
\end{theorem}

Typically to use this bound, one needs to first find suitable parameters $\lambda$ and $\rho_{\ell,\ell'}$ depending on $q$ such that the right-hand side is less than or equal to $2\mathcal{P}(\mu_P)$ for any $q\in[-1,1].$ In Section \ref{sec5}, we shall see that this could be achieved in the case of $\xi_{1,1}=\xi_{2,2}$ and $h_1=h_2$, but the general situation remains mysterious.

\subsection{Some properties of the overlap}\label{subpo}

Recall the Hamiltonian $H_N$ and the Gibbs measure $G_N$ from \eqref{H} and \eqref{Gibbs}. Let $\mu_P$ be the Parisi measure associated to $H_N$ and set $\eta=\min\mbox{supp}\mu_P.$ It is known (see \cite{Chen14}) that 
\begin{align}\label{eta}
\mbox{$\eta=0$ if $h=0$ and $\eta>0$ if $h\neq 0.$}
\end{align}
Recall that as we have discussed in the introduction, the overlap $R_{1,2}$ between two independently sampled spin configurations from $G_N$ is symmetric with respect to the origin if the mixed $p$-spin is even and the external field is absent. The positivity principle of the overlap says that this symmetry will be broken in such a way that the overlap is essentially bounded from below by $\eta$ if the external field is present. More specifically, below is our main result.

\begin{theorem}[Positivity of the overlap]\label{pp}
Assume that $\xi$ is convex on $[-1,1]$ and is not identically equal to zero. If $h\neq 0,$ then under any one of the following two assumptions,
\begin{itemize}
	\item[$(i)$] $\xi$ is even,
	\item[$(ii)$] $\xi$ is not even and the function below is nondecreasing on $(0,1],$
	\begin{align}
	\label{pp:eq2}
	\frac{\xi''(s)}{\xi''(s)+\xi''(-s)},
	\end{align}
\end{itemize} 
we have that for any $\varepsilon>0,$ there exists a constant $K_0>0$ such that
\begin{align}
\label{eq-5}
\e G_N\times G_N\left((\vsi^1,\vsi^2):R_{1,2}\leq \eta-\varepsilon\right)\leq K_0\exp \Bigl(-\frac{N}{K_0}\Bigr),\,\,\forall N\geq 1.
\end{align} 
\end{theorem}

The inequality \eqref{eq-5} means that if the external field is present, then the overlap essentially charges weight only in the interval $[\eta,1]\subseteq (0,1]$. Positivity of the overlap under the condition $(i)$ was initially established by Talagrand \cite[Section 14.10]{Tal11}. Our main contribution here is the case $(ii)$, where we allow odd $p$-spin interactions in the Hamiltonian. Below we describe a concrete example of the case $(ii).$ 

\begin{example}\label{ex3}\rm
Consider $\xi(s)=\beta^2(\gamma_{2p}^{2}s^{2p}+\gamma_{2p+1}^2s^{2p+1})$ on $[-1,1]$ with $\gamma_{2p}$ and $\gamma_{2p+1}$ satisfying
\begin{align*}
c:=\frac{(2p+1)\gamma_{2p+1}^2}{(2p-1)\gamma_{2p}^2}<1.
\end{align*}
It is easy to verify that this condition ensures the convexity of $\xi$ on $[-1,1].$ Since
\begin{align*}
\frac{\xi''(s)}{\xi''(s)+\xi''(-s)}&=\frac{1+cs}{2}
\end{align*}
is nondecreasing on $(0,1]$, condition $(ii)$ in Theorem \ref{pp} is satisfied, from which we obtain \eqref{eq-5} for any $\beta>0.$
\end{example}

Our next result shows that in the absence of the external field $h=0$, the behavior of the overlap is also influenced drastically by the odd $p$-spin interactions in the Hamiltonian, in which case the overlap will be nonnegative. 

\begin{theorem}[Nonnegativity of the overlap]\label{npp}
	 Assume that $\xi$ is convex on $[-1,1]$ and is not identically equal to zero. If $h=0$ and the assumption $(ii)$ in Theorem \ref{pp} holds, then for any $\varepsilon>0,$ there exists a constant $K_0>0$ such that
	\begin{align}
	\label{npp:eq1}
	\e G_N\times G_N\left((\vsi^1,\vsi^2):R_{1,2}\leq -\varepsilon\right)\leq K_0\exp \Bigl(-\frac{N}{K_0}\Bigr),\,\,\forall N\geq 1.
	\end{align} 
\end{theorem}

\subsection{Chaos in disorder}\label{subcd}

Recall the Hamiltonians $H_N^1$ and $H_N^2$ from \eqref{eq0}.  Assume that the Gaussian parts of the Hamiltonians, $X_N^1$ and $X_N^2,$ have the following covariance structure, 
\begin{align}\label{ass1}
\xi_{1,1}=\xi_{2,2}=\xi,\,\,\xi_{1,2}=\xi_{2,1}=\xi_0
\end{align}
for some series $\xi_0$ defined in a similar way as $\xi$ and that the external fields satisfy
\begin{align}\label{ass2}
h_1=h_2=h.
\end{align}
In other words, the two systems have the same distribution and they are coupled through the function $\xi_0.$ Denote by $G_N^1$ and $G_N^2$ the Gibbs measures associated to these Hamiltonians in the same fashion as \eqref{Gibbs}. Note that the two systems share the same Parisi measure $\mu_P$ and $\eta:=\min\mbox{supp}\mu_P$ has the property \eqref{eta}. Consider the overlap $R_{1,2}$ between the independently sampled $\vsi^1$ and $\vsi^2$ from $G_N^1$ and $G_N^2$, respectively. We say that there is chaos in disorder between $H_N^1$ and $H_N^2$ if the overlap is concentrated around a constant value. Our main result shows that this behavior holds as long as the two systems are decoupled, $\xi_0\neq \xi$, for $\xi$ and $\xi_0$ satisfying some mild assumptions:

\begin{theorem}[Disorder chaos]\label{dc}
Assume that \eqref{ass1} and \eqref{ass2} hold. If $\xi$ and $\xi_0$ are convex on $[-1,1]$ and are not identically equal to zero such that 
\begin{align}\label{dc:eq2}
\frac{\xi''(s)}{\xi''(s)+\xi_0''(-s)}\,\,\mbox{and}\,\,\frac{\xi''(s)}{\xi''(s)+\xi_0''(s)} 
\end{align}
{are both nondecreasing on $(0,1]$} and 
\begin{align}
\label{dc:eq3}
\xi_0''(s)<\xi''(|s|)
\end{align} 
for $s\in[-1,1]\setminus\{0\},$ then there is a constant $q^*$ such that for any $\varepsilon>0,$ 
\begin{align}\label{chaos}
\e G_N^1\times G_N^2\left((\vsi^1,\vsi^2):|R_{1,2}-q^*|>\varepsilon\right)\leq K_0\exp \left(-\frac{N}{K_0}\right)
\end{align}
for all $N\geq 1$, where $K_0$ is a constant independent of $N.$ Here, $q^*=0$ if $h=0$ and $q^*\in (0,\eta)$ if $h\neq 0.$

\end{theorem}

Form this theorem, the overlap is basically concentrated around a constant value $q^*$ if the two systems are decoupled in an appropriate way \eqref{dc:eq2} and \eqref{dc:eq3}. We emphasis that this behavior is completely different from the one under the assumption $\xi=\xi_0$, in which case the two systems are indeed identical, $H_N^1=H_N^2=H_N$, and the overlap typically has nontrivial limiting distribution in the low temperature regime. See, for instance, Examples $1$ and $2$ in \cite{Pan08}. The following two choices of $\xi$ and $\xi_0$ summarize the previously known results and give new examples of chaos in disorder. 

\begin{example}[mixed even $p$-spin models]\label{ex4}\rm
	
Assume that the two systems are mixed even $p$-spin models and they are correlated through $\xi_0=t\xi$ for some $t\in(0,1).$ This choice of $(\xi,\xi_0)$ corresponds to \eqref{eq--2} and was originally considered in Chatterjee \cite{Chatt09}, where he proved that the overlap is concentrated around $0$ provided with moment estimates when there is no external field $h=0$. Later Chen \cite{Chen13} established \eqref{chaos} in the presence of external field $h\neq 0.$ One can easily check that $\xi$ and $\xi_0$ are convex functions and \eqref{dc:eq2} and \eqref{dc:eq3} are satisfied. So Theorem \ref{dc} proves disorder chaos irrespective of the presence or absence of the external field.
\end{example}

The main merit of Theorem \ref{dc} is that it also covers the mixed $p$-spin models containing odd $p$-spin interactions for properly chosen sequence $(\gamma_p)_{p\geq 2}$.

\begin{example}
\label{ex5}\rm
Recall $\xi$ and $c$ from Example \ref{ex3}. Let $\xi_{1,1}=\xi_{2,2}=\xi$. For $t\in [0,1)$, set $\xi_0(s):=\beta^2(\gamma_{2p}^{2}s^{2p}+t\gamma_{2p+1}^2s^{2p+1})$ for $s\in[-1,1].$ Since $c<1$ and $t\in [0,1)$, one can check that $\xi_0$ is convex on $[-1,1].$ In addition, since
\begin{align*}
\frac{\xi''(s)}{\xi''(s)+\xi_0''(s)}&=\frac{1+cs}{2\left(1+\left(\frac{1+t}{2}\right)cs\right)},\\
\frac{\xi''(s)}{\xi''(s)+\xi_0''(-s)}&=\frac{1+cs}{2\left(1+\left(\frac{1-t}{2}\right)cs\right)},
\end{align*}
a direct differentiation with respect to $s$ and using $t\in [0,1)$ imply that they are both nondecreasing on $(0,1].$ On the other hand, 
\begin{align*}
\xi_0''(s)&=2p\beta^2 s^{2p-2}\left((2p-1)\gamma_{2p}^2+t(2p+1)\gamma_{2p+1}^2s\right)\\
&<2p\beta^2|s|^{2p-2}\left((2p-1)\gamma_{2p}^2+(2p+1)\gamma_{2p+1}^2|s|\right)=\xi''(|s|)
\end{align*}
for all $s\in [-1,1]\setminus\{0\}.$ Therefore, the conclusion of Theorem \ref{dc} holds for all $\beta>0.$
\end{example}


\section{Directional derivative of the Parisi functional}

The main results stated in Subsection \ref{sub1.1} will be established here. Throughout this section, we will use the variational representation formula \eqref{thm-2:eq1} for $\Phi_{\mu}(0,x)$ with $(s,t)=(0,1)$. Recall the associated maximizer $u_\mu$ from \eqref{max}. We start by computing the directional derivative of the Parisi functional, which relies on two technical lemmas. The first is the combination of \cite[Proposition 2]{AC13} and \cite[Lemm 2]{AC14}. 

\begin{lemma}
\label{sec2:lem2}
For any $\mu\in\mathcal{M}$ and $s\in[0,1]$, $\partial_x\Phi_{\mu}(s,\cdot)$ is odd, strictly increasing and uniformly bounded by $1$. In addition, the process $u_{\mu}$ satisfies 
\begin{align*}
u_{\mu}(b)-u_{\mu}(a)&=\int_a^b\zeta(w)^{1/2}\partial_{xx}\Phi_{\mu}(w,X(w))dB(w)
\end{align*}
for all $0\leq a\leq b\leq 1.$
\end{lemma}

The second lemma allows us to take derivatives for maximum functions under mild assumptions.

\begin{lemma}
\label{lem1}
Let $K$ be a metric space and $I$ be an interval with right open edge. Let $f$ be a real-valued function on $K\times I$ and $g(y)=\sup_{a\in K}f(a,y).$
Suppose that there exists a $K$-valued continuous function $a(y)$ on $I$ such that $g(y)=f(a(y),y)$ and $\partial_yf$ is continuous on $K\times I$, then $g$ is right-differentiable with derivative $\partial_yf(a(y),y)$ for all $y\in I$.
\end{lemma}

\begin{proof}
Let $y\in I.$ Consider any $h>0$ that satisfies $y+h\in I.$ Observe that
\begin{align*}
\frac{g(y+h)-g(y)}{h}&=\frac{f(a(y+h),y+h)-f(a(y),y+h)}{h}+\frac{f(a(y),y+h)-f(a(y),y)}{h}\\
&\geq \frac{f(a(y),y+h)-f(a(y),y)}{h}.
\end{align*}
Therefore, $\liminf_{h\downarrow 0}h^{-1}(g(y+h)-g(y))\geq \partial_y f(a(y),y).$
On the other hand, we also have
\begin{align*}
\frac{g(y+h)-g(y)}{h}&=\frac{f(a(y+h),y+h)-f(a(y+h),y)}{h}+\frac{f(a(y+h),y)-f(a(y),y)}{h}\\
&\leq \frac{f(a(y+h),y+h)-f(a(y+h),y)}{h}\\
&=\partial_y f(a(y+h),y(h))
\end{align*}
for some $y(h)\in I$ with $y(h)\rightarrow y$ as $h\downarrow 0,$ where the last equation used the mean value theorem. Finally, using the continuity of $\partial_yf$, we obtain $
\limsup_{h\downarrow 0}h^{-1}(g(y+h)-g(y))\leq \partial_yf(a(y),y).
$
This finishes our proof.
\end{proof}

\begin{proof}[\bf Proof of Theorem \ref{prop1}] 
Define $$
f(u,\theta)=\log 2+F_{\mu_\theta}^{0,1}(u,h)-\frac{1}{2}\int_0^1\alpha_{\mu_\theta}(s)s\zeta(s)ds
$$
for $(u,\theta)\in D[0,1]\times [0,1].$ Recall the definition of $F_{\mu_\theta}^{0,1}$, 
\begin{align*}
f(u,\theta)&=\log 2+\e\left[\log\cosh\left(h+\int_0^1\alpha_{\mu_\theta}(s)\zeta(s)u(s)ds+\int_0^1\zeta(s)^{1/2}dB(s)\right)\right.\\
&\qquad\qquad\qquad\qquad-\left.\frac{1}{2}\int_0^1\alpha_{\mu_\theta}(s)\zeta(s)(u(s)^2+s)ds\right].
\end{align*}
Its partial derivative with respect to $\theta$ is clearly continuous on $D[0,1]\times[0,1]$ and a direct computation gives
\begin{align*}
\partial_\theta f(u_{\mu_\theta},\theta)
&=\e\left[u_{\mu_\theta}(1)\int_0^1\zeta(s)(\alpha_\mu(s)-\alpha_{\mu_0}(s))u_{\mu_\theta}(s)ds\right.\\
&\left.\qquad\qquad-\frac{1}{2}\int_0^1\zeta(s)(\alpha_\mu(s)-\alpha_{\mu_0}(s))(u_{\mu_\theta}(s)^2+s)ds\right].
\end{align*}
Since $\{u_{\mu_\theta}(r)\}_{0\leq r\leq 1}$ is a martingale from Lemma \ref{sec2:lem2}, the first term can be computed as
\begin{align*}
\int_0^1\zeta(s)(\alpha_\mu(s)-\alpha_{\mu_0}(s))\e u_{\mu_\theta}(s)^2ds
\end{align*}
and thus,
\begin{align*}
\partial_\theta f(u_{\mu_\theta},\theta)
&=\frac{1}{2}\int_0^1\zeta(s)(\alpha_\mu(s)-\alpha_{\mu_0}(s))(\e u_{\mu_\theta}(s)^2-s)ds.
\end{align*}
Applying Lemma \ref{lem1} gives \eqref{prop1:eq1}. From \eqref{prop1:eq1}, if $\mu_0$ is the Parisi measure, then $(ii)$ clearly holds. Assuming $(ii)$, we note that for any $\varepsilon>0,$ there exists some $\delta>0$ such that $\mathcal{P}(\mu_\theta)-\mathcal{P}(\mu_0)\geq -\varepsilon \theta$ whenever $0<\theta<\delta.$ This and the convexity of $\mathcal{P}$ imply
$$\mathcal{P}(\mu_\theta)\leq (1-\theta) \mathcal{P}(\mu_0)+\theta\mathcal{P}(\mu).$$
So
\begin{align*}
&\theta\left(\mathcal{P}(\mu)-\mathcal{P}(\mu_0)\right)=\theta\mathcal{P}(\mu)+(1-\theta)\mathcal{P}(\mu_0)-\mathcal{P}(\mu_0)\geq -\varepsilon \theta.
\end{align*}
Therefore, $\mathcal{P}(\mu)\geq \mathcal{P}(\mu_0)-\varepsilon.$ Since this inequality is true for all $\varepsilon>0,$ we have that $\mathcal{P}(\mu)\geq \mathcal{P}(\mu_0).$ In other words, $\mu_0$ is the minimizer of the Parisi functional and the uniqueness  of the Parisi measure \cite{AC14} implies that $\mu_0=\mu_P$. So $(ii)$ implies $(i).$ Finally, we finish our proof by proving that $(iii)$ yields $(ii).$ Let $\mu\in\mathcal{M}_d^k$ for some $k\geq 0$. Write $\mu=\sum_{p=0}^k a_p\delta_{q_p}$ with $a_p\geq 0$ and $\sum_{p=0}^ka_p=1.$ Define $\mu^p=\delta_{q_p}$ and $\mu_\theta^p=(1-\theta)\mu_0+\theta\mu^p.$ Now applying $(iii)$ to $\mu^p$, we obtain
\begin{align*}
\left.\frac{d}{d\theta}\mathcal{P}(\mu^p_\theta)\right|_{\theta=0}&=\frac{1}{2}\int_0^1\zeta(s)(\alpha_{\mu^p}(s)-\alpha_{\mu_0}(s))(\e u_{\mu_0}(s)^2-s)ds\geq 0
\end{align*}
and thus, using $\sum_{p=0}^ka_p=1$ and $a_p\geq 0,$
\begin{align*}
\left.\frac{d}{d\theta}\mathcal{P}(\mu_\theta)\right|_{\theta=0}&=\frac{1}{2}\int_0^1\zeta(s)(\alpha_\mu(s)-\alpha_{\mu_0}(s))(\e u_{\mu_0}(s)^2-s)ds\\
&=\sum_{p=0}^ka_p\cdot\frac{1}{2}\int_0^1\zeta(s)(\alpha_{\mu^p}(s)-\alpha_{\mu_0}(s))(\e u_{\mu_0}(s)^2-s)ds\\
&=\sum_{p=0}^ka_p\left.\frac{d}{d\theta}\mathcal{P}(\mu^p_\theta)\right|_{\theta=0}\\
&\geq 0.
\end{align*}
Here the second equation used the observation that $\alpha_{\mu}=\sum_{p=0}^ka_p\alpha_{\mu^p}.$
Since this inequality holds for arbitrary probability measures in  $\mathcal{M}_d$, an approximation argument using the definition of the right derivative of $\mathcal{P}$ implies that $\left.\frac{d}{d\theta}\mathcal{P}(\mu_\theta)\right|_{\theta=0}\geq 0$ for all $\mu\in\mathcal{M}$. So we obtain $(ii).$
\end{proof}

\begin{proof}[\bf Proof of Proposition \ref{thm1}] First we claim that \eqref{thm1:eq1} and \eqref{thm1:eq2} hold for $q\in S\cap(0,1).$ Assume $q\in S\cap(0,1)$ is isolated. Define $\mu^1,\mu^2\in\mathcal{M}$ such that
\begin{align*}
\alpha_{\mu^1}(w)&=\left\{
\begin{array}{ll}
\alpha_{\mu_P}(w),&\mbox{if $w\in [0,q-\varepsilon)\cup [q,1]$},\\
\alpha_{\mu_P}(q),&\mbox{if $w\in [q-\varepsilon,q)$},
\end{array}
\right.\\
\alpha_{\mu^2}(w)&=\left\{
\begin{array}{ll}
\alpha_{\mu_P}(w),&\mbox{if $w\in [0,q)\cup [q+\varepsilon,1]$},\\
\alpha_{\mu_P}(q-),&\mbox{if $w\in [q,q+\varepsilon)$}.
\end{array}
\right.
\end{align*}
From Theorem \ref{prop1}, we have
\begin{align}
\begin{split}\label{sec2:eq1}
\left.\frac{d}{d\theta}\mathcal{P}(\mu_\theta^1)\right|_{\theta=0}&=\frac{1}{2}\int_{q-\varepsilon}^q\zeta(r)(\alpha_{\mu_P}(q)-\alpha_{\mu_P}(w))(\e u_{\mu_P}(w)^2-w)dw\geq 0,\\
\left.\frac{d}{d\theta}\mathcal{P}(\mu_\theta^2)\right|_{\theta=0}&=\frac{1}{2}\int_q^{q+\varepsilon}\zeta(r)(\alpha_{\mu_P}(q-)-\alpha_{\mu_P}(w))(\e u_{\mu_P}(w)^2-w)dw\geq 0.
\end{split}
\end{align}
Note that $\alpha_{\mu_P}(q)>\alpha_{\mu_P}(w)$ for $w\in[q-\varepsilon,q)$ and $\alpha_{\mu_P}(q-)<\alpha_{\mu_P}(w)$ for $w\in [q,q+\varepsilon].$
Since $\e u_{\mu_P}(r)^2$ is a continuous function, the inequalities \eqref{sec2:eq1} imply that there exists some $\varepsilon_0>0$ such that $\e u_{\mu_P}(w)^2\geq w$ on $[q-\varepsilon_0,q]$ and $\e u_{\mu_P}(w)^2\leq w$ on $[q,q+\varepsilon_0]$, which clearly gives \eqref{thm1:eq1}. Now suppose that $q$ is an accumulation point of $S\cap(0,1).$ Then there exists $(q_n)_{n\geq 1}\subset S\cap(0,1)$ such that either $q_n\uparrow q$ or $q_n\downarrow q.$ Assuming the first case, we consider $\mu^3,\mu^4\in\mathcal{M}$ defined through
\begin{align*}
\alpha_{\mu^3}(w)&=\left\{
\begin{array}{ll}
\alpha_{\mu_P}(w),&\mbox{if $w\in [0,q-\varepsilon)\cup [q,1]$},\\
\alpha_{\mu_P}(q),&\mbox{if $w\in [q-\varepsilon,q)$},
\end{array}
\right.\\
\alpha_{\mu^4}(w)&=\left\{
\begin{array}{ll}
\alpha_{\mu_P}(w),&\mbox{if $w\in [0,q-\varepsilon)\cup [q,1]$},\\
\alpha_{\mu_P}(q-\varepsilon),&\mbox{if $w\in [q-\varepsilon,q)$}.
\end{array}
\right.
\end{align*}
From Theorem \ref{prop1}, we have
\begin{align}
\begin{split}\label{sec2:eq3}
\left.\frac{d}{d\theta}\mathcal{P}(\mu_\theta^3)\right|_{\theta=0}&=\frac{1}{2}\int_{q-\varepsilon}^q\zeta(w)(\alpha_{\mu_P}(q)-\alpha_{\mu_P}(w))(\e u_{\mu_P}(w)^2-w)dw\geq 0,
\end{split}\\
\begin{split}\label{sec2:eq4}
\left.\frac{d}{d\lambda}\mathcal{P}(\mu_\theta^4)\right|_{\lambda=0}&=\frac{1}{2}\int_{q-\varepsilon}^{q}\zeta(w)(\alpha_{\mu_P}(q-\varepsilon)-\alpha_{\mu_P}(w))(\e u_{\mu_P}(w)^2-w)dw\geq 0.
\end{split}
\end{align}
From the condition $q_n\uparrow q,$ we see that $\alpha_{\mu_P}(q)>\alpha_{\mu_P}(w)$ and $\alpha_{\mu_P}(q-\varepsilon)<\alpha_{\mu_P}(w)$ for  $w\in [q-\varepsilon,q).$ The inequality \eqref{sec2:eq3} then gives $\e u_{\mu_P}(w)^2\geq w$ for all $w$ sufficiently close to $q$ from the left-hand side. On the other hand, since $q_n\uparrow q$, the inequality \eqref{sec2:eq4} implies that $\e u_{\mu_P}(w)^2\leq w$ for all $w$ sufficiently close to $q$ again from the left-hand side. Therefore, $\e u_{\mu_P}(w)^2=w$ on $[q-\varepsilon_0',q]$ for some $\varepsilon_0'>0$. Similarly, the case $q_n\downarrow q$ also implies $\e u_{\mu_P}(w)^2=w$ on $[q,q+\varepsilon_0'']$ for some $\varepsilon_0''>0$ by using 
\begin{align*}
\alpha_{\mu^5}(w)&=\left\{
\begin{array}{ll}
\alpha_{\mu_P}(w),&\mbox{if $w\in [0,q)\cup [q+\varepsilon,1]$},\\
\alpha_{\mu_P}(q),&\mbox{if $w\in [q,q+\varepsilon)$},
\end{array}
\right.\\
\alpha_{\mu^6}(w)&=\left\{
\begin{array}{ll}
\alpha_{\mu_P}(w),&\mbox{if $w\in [0,q)\cup [q+\varepsilon,1]$},\\
\alpha_{\mu_P}(q+\varepsilon),&\mbox{if $w\in [q,q+\varepsilon)$}.
\end{array}\right.
\end{align*}
These yield \eqref{thm1:eq1}.
To show \eqref{thm1:eq2}, we note that from Lemma \ref{sec2:lem2},
\begin{align}
\begin{split}\label{sec2:eq2}
\e u_{\mu_P}(b)^2-\e u_{\mu_P}(a)^2&=\int_a^b\zeta(r)\e\partial_{xx}\Phi_{\mu_P}(w,X(w))^2dw.
\end{split}
\end{align}
From the discussion,above, we see that either $\e u_{\mu_P}(w)^2\leq w$ on $[q,q']$ for some $q'>q$ or $\e u_{\mu_P}(w)^2\geq w$ on $[q'',q]$ for some $q''<q.$ If we are in the first situation, then for all $s\in[q,q'],$ \eqref{sec2:eq2} implies
\begin{align*}
\int_q^s\zeta(w)\e\partial_{xx}\Phi_{\mu_P}(w,X(w))^2dw=\e u_{\mu_P}(s)^2-\e u_{\mu_P}(q)^2\leq s-q=\int_{q}^s 1dw.
\end{align*}
and hence \eqref{thm1:eq2}. In the second situation, the same argument also yields
\begin{align*}
\int_s^q\zeta(w)\e\partial_{xx}\Phi_{\mu_P}(w,X(w))^2dw= \e u_{\mu_P}(q)^2-\e u_{\mu_P}(s)^2\leq q-s=\int_{s}^q 1dw
\end{align*}
for all $s\in[q'',q]$, which concludes \eqref{thm1:eq2} and completes the proof of our claim. 

Finally, note that Lemma \ref{sec2:lem2} yields $\e u_{\mu_P}(r)^2<1$ for all $0\leq r\leq 1$. If $1\in S,$ one may take $\mu=\delta_0$ and $\mu_0=\mu_P$ in \eqref{prop1:eq1} to obtain a contradiction since $\frac{d}{d\theta}\mathcal{P}(\mu_\theta)|_{\theta=0}<0.$ Hence, $1\notin S.$ If now $0\in S,$ then no matter it is an isolated point or an accumulation point of $S$, one can argue exactly in the same way as above to obtain $\e u_{\mu_P}(w)^2\leq w$ for all $w\in[0,\varepsilon_0]$ for some $\varepsilon_0>0.$ Consequently, $\e u_{\mu_P}(0)^2=0$. Since
\begin{align*}
\int_0^s\zeta(w)\e\partial_{xx}\Phi_{\mu_P}(w,X(w))^2dw=\e u_{\mu_P}(s)^2-\e u_{\mu_P}(0)^2\leq s-0=\int_{0}^s 1dw
\end{align*}
for all $s\in [0,\varepsilon_0]$,
we obtain \eqref{thm1:eq2} with $q=0.$ This finishes our proof. \end{proof}

\begin{proof}[\bf Proof of Proposition \ref{prop4}]
For any $q\in[0,1],$ define $\mu_\theta^q=(1-\theta)\mu_0+\theta \delta_q$ for $0\leq \theta\leq 1.$ Since $\mu_\theta^q\in \mathcal{M}_d^{k+1},$ it follows from \eqref{prop4:eq1} that $\left.\frac{d}{d\theta}\mathcal{P}(\mu_\theta^q)\right|_{\theta=0}\geq 0.$ Therefore, $\mu_0$ is the Parisi measure by applying Theorem \ref{prop1} $(ii).$
\end{proof}


\section{The optimal stochastic control problem for $\Psi_\mu$}

In this section, we will prove Theorems \ref{sec1.2:thm1}, \ref{thm2} and \ref{thm-1} following the ideas mostly from \cite{AC14}. Our argument relies on the following calculus lemma, which provides an explicit expression for the function $\Psi_\mu$ when $\mu\in \mathcal{M}_d$. As this lemma is not directly related to the core of our arguments, we defer its proof to Appendix. 

\begin{lemma}\label{lem3}
Let $0\leq a<b\leq 1$ and $0\leq m\leq 1$.  Recall $\rho_{\ell,\ell'}$ from \eqref{eq-10}. Suppose that they are differentiable on $[a,b)$.
Let $A$ be a smooth function on $\mathbb{R}^2$ with $\limsup_{|\vx|\rightarrow\infty}|A(\vx)|/|\vx|<\infty.$ For $(s,\vx)\in[a,b]\times\mathbb{R}^2$, set 
\begin{align*}
L(s,\vx)&=\frac{1}{m}\log \e\exp mA\left(x_1+y_1(s),x_2+y_2(s)\right),
\end{align*}
where $(y_1(s),y_2(s))$ is a two-dimensional Gaussian random vector with mean zero and covariance,
\begin{align*}
\e y_{1}(s)y_{1}(s)&=\xi_{1,1}'(\rho_{1,1}(b))-\xi_{1,1}'(\rho_{1,1}(s)),\\
\e y_{1}(s)y_{2}(s)&=\xi_{1,2}'(\iota\rho_{1,2}(b))-\xi_{1,2}'(\iota\rho_{1,2}(s)),\\
\e y_{2}(s)y_{1}(s)&=\xi_{2,1}'(\iota\rho_{2,1}(b))-\xi_{2,1}'(\iota\rho_{2,1}(s)),\\
\e y_{2}(s)y_{2}(s)&=\xi_{2,2}'(\rho_{2,2}(b))-\xi_{2,2}'(\rho_{1,1}(s)).
\end{align*}
Then $L$ satisfies 
\begin{align}\label{lem3:eq1}
\partial_sL&=-\frac{1}{2}\left(\left<T ,\triangledown^2L\right>+m\left<T \triangledown L,\triangledown L\right>\right)
\end{align}
for $(s,\vx)\in [a,b)\times\mathbb{R}^2$ with terminal condition $L(b,\vx)=A(\vx).$ Moreover, if $\partial_{x_i}A$ is uniformly bounded by $1$, so is $\partial_{x_i}L$.
\end{lemma}

\begin{proof}[\bf Proof of Theorem \ref{thm2} for $\boldsymbol{\mu\in\mathcal{M}_d}$] Suppose that $\mu$ is atomic with jumps at $\{q_p\}_{p=1}^k$, where $q_p<q_{p+1}$ for $1\leq p\leq k-1$. Let $q_0=0$, $q_{k+1}=1$ and $m_p=\alpha_{\mu}(q_p)$ for $0\leq p\leq k.$ Without loss of generality, we may assume that the non-differentiable points of $\rho_{\ell,\ell'}$ are located at $\{q_p\}_{p=1}^k$ and in addition, $q_j=s$ and $q_{j'}=t$ for some $0\leq j<j'\leq k+1.$ Note that since $(y_p^1(s),y_p^2(s))$ equals $\int_s^{q_{p+1}}T(w)^{1/2}d\mathcal{B}(w)$ in distribution for each $s\in [q_p,q_{p+1}]$ and $0\leq p\leq k,$ we can write by Lemma \ref{lem3},
\begin{align}\label{proof:thm2:eq1}
\Psi_\mu(\lambda,q_p,x)&=\frac{1}{m_p}\log\e\exp m_p\Psi_\mu\Bigl(\lambda,q_{p+1},x+\int_{q_p}^{q_{p+1}}T(w)^{1/2}d\mathcal{B}(w)\Bigr),\,\,\forall j\leq p<j'.
\end{align}
We claim that
\begin{align}\label{eq-1}
\Psi_{\mu}(\lambda,s,\vx)&\geq \max_{v\in\mathcal{D}[s,t]}\bF_{\mu}^{s,t}(\lambda,v,\vx).
\end{align}
For $v\in\mathcal{D}[s,t],$ set
\begin{align*}
Z_p&=\exp\Bigl(-\frac{1}{2}\int_{q_p}^{q_{p+1}}m_p^2\la T(w)v(w),v(w)\ra dw-\int_{q_q}^{q_{p+1}}m_p\la T(w)^{1/2}v(w), d\mathcal{B}(w)\ra\Bigr).
\end{align*}
Define conditional probability measure $\tilde{\p}(A)=\e [1_AZ_p|\mathscr{G}_{q_p}]$ and set $\tilde{\mathcal{B}}(r)=\int_{q_p}^{r}m_pT(w)^{1/2}v(w)dw+\mathcal{B}(r)$ for $r\in[q_{p},q_{p+1}].$ We use $\tilde{\e}$ to denote the expectation with respect to $\tilde{\p}.$
Since the Girsanov theorem \cite[Theorem 5.1]{KS} says that $\tilde{\mathcal{B}}$ is a standard Brownian motion starting from $\mathcal{B}(q_p)$ under $\tilde{\mathbb{P}}$, we can write
\begin{align*}
&\e\exp m_p\Psi_\mu\Bigl(\lambda,q_{p+1},\vx+\int_{q_p}^{q_{p+1}}T(w)^{1/2}d\mathcal{B}(w)\Bigr)\\
&=\tilde{\e}\exp m_p\Psi_\mu\Bigl(\lambda,q_{p+1},\vx+\int_{q_p}^{q_{p+1}}T(w)^{1/2}d\tilde{\mathcal{B}}(w)\Bigr)\\
&=\e\Bigl[\exp m_p\Psi_\mu\Bigl(\lambda,q_{p+1},\vx+\int_{q_p}^{q_{p+1}}m_pT(w)v(w)dw+\int_{q_p}^{q_{p+1}}T(w)^{1/2}d\mathcal{B}(w)\Bigr)\\
&\qquad\cdot \exp\Bigl(-\frac{1}{2}\int_{q_p}^{q_{p+1}}m_p^2\la T(w)v(w),v(w)\ra dw-\int_{q_q}^{q_{p+1}}m_pT(w)^{1/2}v(w)\cdot d\mathcal{B}(w)\Bigr)\Big|\mathscr{G}_{q_p}\Bigr].
\end{align*}
From \eqref{proof:thm2:eq1} and Jensen's inequality $m^{-1}\log \e [\exp mA|\mathscr{G}_{q_p}]\geq \e[A|\mathscr{G}_{q_p}]$ for any measurable $A$ and $m>0$, it follows
\begin{align*}
\Psi_\mu(\lambda,q_p,\vx)
&\geq\e \Bigl[\Psi_\mu\Bigl(\lambda,q_{p+1},\vx+\int_{q_p}^{q_{p+1}}\alpha_{\mu}(w)T(w)v(w)dw+\int_{q_p}^{q_{p+1}}T(w)^{1/2}d\mathcal{B}(w)\Bigr)\\
&\qquad\qquad\qquad
-\frac{1}{2}\int_{q_p}^{q_{p+1}}\alpha_\mu(w)\la T(w)v(w),v(w)\ra dw\Big|\mathscr{G}_{q_p}\Bigr]
\end{align*}
for all $j\leq p<j'.$ Using this and conditional expectation, a decreasing iteration argument over $p$ from $j'-1$ to $j$ gives
\begin{align*}
\Psi_\mu(\lambda,s,\vx)&=\Psi_\mu(\lambda,q_j,\vx)\\
&\geq \e\Bigl[\Psi_\mu\Bigl(\lambda,q_{j'},\vx+\sum_{p=j}^{j'-1}\int_{q_p}^{q_{p+1}}\alpha_\mu(w)T(w)v(w)dw+\sum_{p=j}^{j'-1}\int_{q_p}^{q_{p+1}}\zeta(w)^{1/2}d\mathcal{B}(w)\Bigr)\\
&\qquad\qquad\qquad
-\frac{1}{2}\sum_{p=j}^{j'-1}\int_{q_p}^{q_{p+1}}\alpha_\mu(w)\la T(w)v(w),v(w)\ra dw\Bigr]\\
&=\bF_{\mu}^{q_j,q_{j'}}(\lambda,v,\vx)\\
&=\bF_\mu^{s,t}(\lambda,v,\vx).
\end{align*}
Since this is true for arbitrary $v\in\mathcal{D}[s,t],$ this gives \eqref{eq-1}. 

Note that since $|\partial_{x_1}\Psi_{\mu}(\lambda,1,\cdot)|$ and $|\partial_{x_2}\Psi_{\mu}(\lambda,1,\cdot)|$ are uniformly bounded above by $1,$ Lemma \ref{lem3} combined with an iteration argument using \eqref{proof:thm2:eq1} yields that $|\partial_{x_1}\Psi_{\mu}(\lambda,r,\cdot)|$ and $|\partial_{x_2}\Psi_{\mu}(\lambda,r,\cdot)|$ are also uniformly bounded by $1$ for any $s\leq r\leq t$, which clearly imply that $v_\mu\in \mathcal{D}[s,t].$ Therefore, to finish the proof, it remains to show that $\mathcal{F}_\mu^{s,t}(\lambda,v_\mu,\vx)=\Psi_\mu(s,\vx).$ To this end, we define
\begin{align*}
Y(r)&=\Psi_\mu(\lambda,r,\vX(r))-\int_{s}^r\alpha_{\mu}(w)\left<T(w)v_\mu(w),v_\mu(w)\right>dw-\int_s^rT(w)^{1/2}d\mathcal{B}(w).
\end{align*}
Observe that
\begin{align*}
\e Y(s)&=\e\Psi_\mu(\lambda,s,\vX(s))=\Psi_\mu(\lambda,s,\vx),\\
\e Y(t)&=\bF_\mu^{s,t}(\lambda,v_\mu,\vx).
\end{align*}
The use of It\^{o}'s formula and \eqref{eq-9} implies
\begin{align*}
d\Psi_\mu&=\partial_s\Psi_\mu dw+\left<\triangledown \Psi_\mu,d\vX\right>+\frac{1}{2}\sum_{i,j=1}^2\partial_{x_ix_j}\Psi_\mu d\left<X_i,X_j\right>\\
&=-\frac{1}{2}\left(\left<T,\triangledown^2\Psi_\mu\right>+\alpha_\mu \left<T\triangledown \Psi_\mu,\triangledown \Psi_\mu\right>\right)dw\\
&+\alpha_\mu\zeta \left<T\triangledown\Psi_\mu,\triangledown\Psi_\mu\right>dw+T^{1/2}\left<\triangledown \Psi_\mu,d\mathcal{B}\right>+\frac{1}{2}\left<T,\triangledown^2\Psi_\mu\right>dw\\
&=\frac{1}{2}\alpha_\mu\left<T\triangledown\Psi_\mu,\triangledown\Psi_\mu\right>dw+T^{1/2}\left<\triangledown \Psi_\mu,d\mathcal{B}\right>
\end{align*}
and thus, $dY=0$, which means that $\bF_\mu^{s,t}(\lambda,v_\mu,\vx)=\e Y(t)=\e Y(s)=\Psi_\mu(\lambda,s,\vx)$. This finishes our proof. 
\end{proof}

\begin{proof}[\bf Proof of Theorem \ref{sec1.2:thm1}]
Let $\mu,\mu'\in\mathcal{M}_d$. Since
\begin{align*}
\Psi_{\mu}(\lambda,1,\vx)&=\Psi_{\mu'}(\lambda,1,\vx)=\log(\cosh x_1\cosh x_2\cosh \lambda+\sinh x_1\sinh x_2\sinh\lambda),
\end{align*}
the mean value theorem implies
$$
|\Psi_{\mu}(\lambda,1,\vx)-\Psi_{\mu'}(\lambda,1,\vx')|\leq |\vx-\vx'|
$$
for $\lambda\in\mathbb{R}$ and $\vx,\vx'\in\mathbb{R}^2$. Therefore, for any $v\in\mathcal{D}[s,1],$
\begin{align*}
\left|\mathcal{C}_{\mu}^{s,1}(\lambda,v,\vx)-\mathcal{C}_{\mu'}^{s,1}(\lambda,v,\vx)\right|
&\leq\int_0^1|\alpha_{\mu}(w)-\alpha_{\mu'}(w)||T(w)v(w)|dw\\
&\leq \sqrt{2}Kd(\mu,\mu'),
\end{align*}
where the last inequality used $\|T(w)\|\leq K$ and $|v(w)|\leq \sqrt{2}.$
Also, we know that
\begin{align*}
|\mathcal{L}_\mu^{s,1}(v)-\mathcal{L}_{\mu'}^{s,1}(v)|&\leq Kd(\mu,\mu').
\end{align*}
Combining these two inequalities together leads to
\begin{align*}
\left|\mathcal{F}_{\mu}^{s,1}(\lambda,v,\vx)-\mathcal{F}_{\mu'}^{s,1}(\lambda,v,\vx)\right|&\leq 3Kd(\mu,\mu')
\end{align*}
and hence the announced inequality by applying \eqref{thm2:eq2}.
\end{proof}

\begin{proof}[\bf Proof of Theorem \ref{thm2} for arbitrary $\boldsymbol{\mu}$] This part of the proof relies on a standard approximation by using a sequence of atomic $\{\mu_n\}_{n\geq 1}$ with weak limit $\mu.$ Just like the facts that $\partial_{x^i}\Phi_{\mu}$ is uniformly bounded by $1$ and $\lim_{n\rightarrow\infty}\partial_{x^i}\Phi_{\mu_n}=\partial_{x^i}\Phi_{\mu}$ uniformly for $i=1,2$, one may imitate the same approach as the appendix in \cite{AC13} to show that $|\partial_{x_i}\Psi_{\mu}|\leq 1$, $\|\triangledown^2\Psi_{\mu}\|\leq C$ and $\lim_{n\rightarrow\infty}\triangledown^i\Psi_{\mu_n}=\triangledown^i\Psi_{\mu}$ uniformly for $i=1,2.$ These give the existence of the SDE \eqref{thm2:eq1} and will lead to $(i)$ and $(ii)$ by using the results for atomic measures we established above and the same argument as in the proof of \cite[Theorem 3]{AC14}. As the details are quite routine and follow exactly in the same lines, we will not reproduce them here.
\end{proof}

\begin{proof}[\bf Proof of Theorem \ref{thm-1}]
By the virtue of the Lipschitz property of $\mu\mapsto\Psi_\mu$ with respect to the metric $d$ defined by \eqref{metric}, it suffices to justify \eqref{thm-1:eq1} for atomic $\mu$ with jumps at $\{q_p\}_{p=1}^{k}$, where $q_p<q_{p+1}$ for all $1\leq p\leq k-1.$ Let $q_0=0$ and $q_{k+1}=1.$ Without loss of generality, we may also assume that the non-differentiable points of $\rho_{\ell,\ell'}$ are all at $\{q_p\}_{p=1}^{k}.$ Set 
\begin{align*}
\rho_p^{1,1}&=\rho_{1,1}(q_p),\,\,\rho_p^{2,2}=\rho_{2,2}(q_p),\,\,
\rho_p^{1,2}=\iota\rho_{1,2}(q_p),\,\,\rho_p^{2,1}=\iota\rho_{2,1}(q_p)
\end{align*}
for $0\leq p\leq k+1.$ Note that \eqref{eq-2} follows from \eqref{eq--1}. Since 
\begin{align*}
\xi_{\ell,\ell'}'(\rho_{p+1}^{\ell,\ell'})-\xi_{\ell,\ell'}'(\rho_p^{\ell,\ell'})
&=\int_{q_p}^{q_{p+1}}\zeta_{\ell,\ell'}(s)ds,
\end{align*}
the assumption $T(s)\geq 0$ implies that
\begin{align*}
\left<\left[
\begin{array}{cc}
\xi_{1,1}'(\rho_{p+1}^{1,1})-\xi_{1,1}'(\rho_p^{1,1})&\xi_{1,2}'(\rho_{p+1}^{1,2})-\xi_{1,2}'(\rho_p^{1,2})\\
\xi_{2,2}'(\rho_{p+1}^{2,2})-\xi_{2,2}'(\rho_p^{2,2})&\xi_{2,1}'(\rho_{p+1}^{2,1})-\xi_{2,1}'(\rho_p^{2,1})
\end{array}
\right]\vx,\vx\right>
&=\int_{q_p}^{q_{p+1}}\left<T(s)\vx,\vx\right>ds\geq 0
\end{align*}
for all $\vx\in\mathbb{R}^2$. So the matrix on the right-hand side is positive semi-definite, which ensures that we can construct Gaussian random vectors $(y_p^{\ell},y_p^{\ell'})$ with mean zero and covariance
$$
\e y_p^{\ell}y_p^{\ell'}=\xi_{\ell,\ell'}'(\rho_{p+1}^{\ell,\ell'})-\xi_{\ell,\ell'}'(\rho_{p}^{\ell,\ell'}).
$$
Now we apply Theorem \ref{thm0} with the choice $m_p=\mu([0,q_{p}])$ for $0\leq p\leq k$ to get \eqref{thm0:eq1} as follows. Recall the definition of $Y_{p}$ for $0\leq p\leq k+1$ from Theorem \ref{thm0}. Define
\begin{align*}
(Z_p^1,Z_p^2)&=\bigg(h_1+\sum_{j=0}^{p-1}y_j^1,h_2+\sum_{j=0}^{p-1}y_j^2\bigg),\,\,\forall 1\leq p\leq k+1,\\
(Z_0^1,Z_0^2)&=(h_1,h_2).
\end{align*}
Observe that $Y_{k+1}=\Psi_{\mu}(\lambda,1,Z_{k+1}^1,Z_{k+1}^2)$. If $Y_{p+1}=\Psi_{\mu}(\lambda,q_{p+1},Z_{p+1}^1,Z_{p+1}^2)$ for some $0\leq p\leq k$, then Lemma \ref{lem3} yields
\begin{align*}
Y_p&=\frac{1}{m_p}\log \e_p\exp m_pY_{p+1}\\
&=\frac{1}{m_p}\log \e_p\exp m_p\Psi_{\mu}(\lambda,q_{p+1},Z_{p}^1+y_{p}^1,Z_p^2+y_p^2)\\
&=\Psi_{\mu}(\lambda,q_p,Z_p^1,Z_p^2)
\end{align*}
and so $Y_0=\Phi_{\mu}(\lambda,0,h_1,h_2).$ On the other hand, since  $\theta_{\ell,\ell'}'(w)=w\xi_{\ell,\ell'}''(w)$, we have that for $\ell=\ell',$
\begin{align*}
\theta_{\ell,\ell'}(\rho_{p+1}^{\ell,\ell'})-\theta_{\ell,\ell'}(\rho_{p}^{\ell,\ell'})
&=\theta_{\ell,\ell'}(\rho_{\ell,\ell'}(q_{p+1}))-\theta_{\ell,\ell'}(\rho_{\ell,\ell'}(q_p))\\
&=\int_{q_p}^{q_{p+1}}\rho_{\ell,\ell'}'(s)\theta_{\ell,\ell'}'(\rho_{\ell,\ell'}(s))ds=\int_{q_p}^{q_{p+1}}\rho_{\ell,\ell'}(s)\zeta_{\ell,\ell'}(s)ds
\end{align*}
and for $\ell\neq \ell,$
\begin{align*}
\theta_{\ell,\ell'}(\rho_{p+1}^{\ell,\ell'})-\theta_{\ell,\ell'}(\rho_{p}^{\ell,\ell'})
&=\theta_{\ell,\ell'}(\iota\rho_{\ell,\ell'}(q_{p+1}))-\theta_{\ell,\ell'}(\iota\rho_{\ell,\ell'}(q_p))\\
&=\int_{q_p}^{q_{p+1}}\iota\rho_{\ell,\ell'}'(s)\theta_{\ell,\ell'}'(\iota\rho_{\ell,\ell'}(s))ds=\int_{q_p}^{q_{p+1}}\iota\rho_{\ell,\ell'}(s)\zeta_{\ell,\ell'}(s)ds.
\end{align*}
Consequently, 
$$
\sum_{1\leq \ell,\ell'\leq 2}\sum_{p=0}^km_p(\theta_{\ell,\ell'}(\rho_{p+1}^{\ell,\ell'})-\theta_{\ell,\ell'}(\rho_{p}^{\ell,\ell'}))=\int_0^1\alpha_{\mu}(s)
\Biggl(\sum_{\ell=\ell'}\rho_{\ell,\ell'}(s)\zeta_{\ell,\ell'}(s)+\iota\sum_{\ell\neq\ell'}\rho_{\ell,\ell'}(s)\zeta_{\ell,\ell'}(s)\Biggr)ds.
$$
Putting all these together into \eqref{thm0:eq1}, we obtain \eqref{thm-1:eq1}.
\end{proof}

\section{The control of the GT bound}\label{sec5}
This section is devoted to proving Theorems \ref{pp} and \ref{dc} in Subsections \ref{subpo} and \ref{subcd}. We assume throughout this section that $X_N^1$ and $X_N^2$ are jointly Gaussian processes with mean zero and covariance,
\begin{align*}
\e X_N^1(\vsi^1)X_N^1(\vsi^2)&=N\xi(R_{1,2}),\\
\e X_N^2(\vsi^1)X_N^2(\vsi^2)&=N\xi(R_{1,2}),\\
\e X_N^1(\vsi^1)X_N^2(\vsi^2)&=N\xi_0(R_{1,2}),
\end{align*}
where $\xi$ and $\xi_0$ are of the form \eqref{eq-4}. Furthermore, we assume that they are convex on $[-1,1]$ and are not identically equal to zero such that 
\begin{align}
\label{cond}
\xi_0''(s)\leq \xi''(|s|),\,\,\forall s\in[-1,1].
\end{align} 
Let $h\in\mathbb{R}$. Consider two mixed $p$-spin models, 
\begin{align*}
H_N^1(\vsi^1)=X_N^1(\vsi^1)+h\sum_{i=1}^N\sigma_i^1\,\,\mbox{and}\,\,H_N^2(\vsi^2)=X_N^2(\vsi^2)+h\sum_{i=1}^N\sigma_i^2.
\end{align*}
Clearly they share the same Parisi measure $\mu_P$. Denote by $\eta$ the minimum of the support of $\mu_P.$ Recall the formulation of the two-dimensional GT bound from \eqref{thm-1:eq1}. For fixed $q\in S_N$, set
\begin{align*}
\rho_{1,1}(s)&=\rho_{2,2}(s)=s,\\
\rho_{1,2}(s)&=\rho_{2,1}(s)=\min(|q|,s)
\end{align*}
for $s\in[0,1]$. From \eqref{eq-10}, it follows that 
\begin{align}\label{eq-12}
T(s)=\left[
\begin{array}{cc}
\xi''(s)&\iota\xi_0''(\iota s)\\
\iota\xi_0''(\iota s)&\xi''(s)
\end{array}
\right],\,\,\forall s\in [0,|q|)\,\,\mbox{and}\,\,
T(s)=\left[
\begin{array}{cc}
\xi''(s)&0\\
0&\xi''(s)
\end{array}\right],\,\,\forall s\in[|q|,1].
\end{align}
Consequently, from the condition \eqref{cond}, one sees that $T\geq 0$ on $[0,|q|);$ also it is clear that $T\geq 0$ on $[|q|,1].$
These allow us to apply Theorem \ref{thm-1} with arbitrary $\mu\in\mathcal{M}$ to get
\begin{align}
\label{eq-11}
F_N(q)&\leq 2\log 2+\Psi_{\mu}(\lambda,0,h,h)-\lambda q-\left(\int_0^1\alpha_{\mu}(s)s\xi''(s)ds+\int_0^{|q|}\alpha_{\mu}(s)s\xi_0''(\iota s)ds\right).
\end{align}
Note that the right-hand side of this inequality is indeed well-defined for all $q\in[-1,1]$. We denote this extension by $\Lambda_\mu(\lambda,q)$ and set $\Lambda(q)=\inf_{\lambda\in\mathbb{R},\mu\in\mathcal{M}}\Lambda_\mu(\lambda,q)$. In the following two subsections, we will control $\Lambda(q)$ using the GT bound in two disjoint regions: $[-\eta,\eta]$ and $[-1,-\eta)\cup(\eta,1]$.

\subsection{Behavior of $\Lambda$ in $[-\eta,\eta]$}

The main result in this subsection is Proposition \ref{prop0} below. This part of the argument appeared before in \cite{Chen14} and \cite[Chapter 14]{Tal11}. For completeness, we will give the detailed proof in the terminology of the variational representation \eqref{thm-2:eq1} and \eqref{thm2:eq2}. Recall that $\eta$ satisfies \eqref{eta}.

\begin{proposition}
\label{prop0} 
If $h\neq 0$, then there exists some $q^*\in (0,\eta]$ such that $$
\Lambda(q)<2\mathcal{P}(\mu_P)$$ 
for any $q\in[-\eta,\eta]\setminus\{q^*\}$. Here, $q^*=\eta$ if $\xi=\xi_0$ and $q^*<\eta$ if $\xi\neq \xi_0.$
\end{proposition}

The proof of this proposition relies on the following technical lemma.

\begin{lemma}
\label{lem2} Let $s\in [|q|,1]$ and $\vx=(x_1,x_2)\in\mathbb{R}.$ If $\mu,\mu'\in\mathcal{M}$ satisfies $\mu=\mu'$ on $[|q|,1]$, then
\begin{align}
\begin{split}\label{eq-14}
\Psi_{\mu}(0,s,\vx)&=\Phi_{\mu'}(s,x_1)+\Phi_{\mu'}(s,x_2),
\end{split}\\
\begin{split}
\label{eq-15}
\partial_\lambda\Psi_\mu(0,s,\vx)&=\partial_{x}\Phi_{\mu'}(s,x_1)\partial_x\Phi_{\mu'}(s,x_2).
\end{split}
\end{align}
\end{lemma}

\begin{proof} For any $|q|\leq s\leq 1$ and $v=(v_1,v_2)\in\mathcal{D}[s,1]$, we write by \eqref{eq-12},
\begin{align}
\begin{split}\label{lem2:proof:eq1}
&\vx+\int_{s}^1\alpha_\mu(r)T(r)v(r)dr+\int_{s}^1T(r)^{1/2}d\mathcal{B}(r)\\
&=\left(x_1+\int_{s}^1\alpha_{\mu'}(r)\xi''(r)v_1(r)dr+\int_{s}^1\xi''(r)^{1/2}d\mathcal{B}_1(r),\right.\\
&\qquad\left.x_2+\int_{s}^1\alpha_{\mu'}(r)\xi''(r)v_2(r)dr+\int_{s}^1\xi''(r)^{1/2}d\mathcal{B}_2(r)\right)
\end{split}
\end{align}
and
\begin{align}\label{lem2:proof:eq2}
\int_{s}^1\alpha_\mu(r)\left<T(r)v(r),v(r)\right>dr&=\int_{s}^1\alpha_{\mu'}(r)\xi''(r)v_1(r)^2dr+\int_{s}^1\alpha_{\mu'}(r)\xi''(r)v_2(r)^2dr.
\end{align}
From the terminal condition of $\Psi_{\mu}$ at \eqref{tc},
\begin{align}
\begin{split}\label{lem2:proof:eq3}
\Psi_{\mu}(0,1,\vx)&=\log\cosh x_1+\log\cosh x_2=\Phi_{\mu'}(1,x_1)+\Phi_{\mu'}(1,x_2),
\end{split}\\
\begin{split}\label{lem2:proof:eq4}
\partial_\lambda\Psi_{\mu}(0,1,\vx)&=\tanh x_1\cdot\tanh x_2=\partial_{x}\Phi_{\mu'}(1,x_1)\cdot\partial_x\Phi_{\mu'}(1,x_2).
\end{split}
\end{align}
Using \eqref{thm-2:eq1} and \eqref{thm2:eq2}, the equations \eqref{lem2:proof:eq1}, \eqref{lem2:proof:eq2} and \eqref{lem2:proof:eq3} yield \eqref{eq-14} since
\begin{align*}
\Psi_{\mu}(0,s,\vx)
&=\max_{v=(v_1,v_2)\in\mathcal{D}[{s},1]}\mathcal{F}_{\mu}^{{s},1}(0,v,\vx)\\
&=\max_{v_1\in D[{s},1]}F_{\mu'}^{{s},1}(v_1,x_1)+\max_{v_2\in D[{s},1]}F_{\mu'}^{{s},1}(v_2,x_2)\\
&=\Phi_{\mu'}({s},x_1)+\Phi_{\mu'}({s},x_2).
\end{align*}
To show \eqref{eq-15}, let $v_\mu(r)=\triangledown\Psi_\mu(0,r,\vX(r))$ be the maximizer for $\Psi_{\mu}(0,s,\vx)$, where $\vX(r)=(X_1(r),X_2(r))$ follows \eqref{thm2:eq1}. The key observation is that the use of \eqref{eq-14} leads to
\begin{align*}
X_i(r)&=x_i+\int_{s}^r\alpha_{\mu'}(w)\xi''(w)\partial_x\Phi_{\mu'}(w,X_i(w))dw+\int_{s}^r\xi''(w)^{1/2}d\mathcal{B}_i(w)
\end{align*}
for $i=1,2.$ Therefore, $\Phi_{\mu'}(s,X_i(s))$ is the maximizer of \eqref{thm-2:eq1} and $\partial_x\Phi_{\mu'}(s,x_i)=\e\partial_x\Phi_{\mu'}(1,X_i(1))$ from Lemma \ref{lem1}. 
Using these and Lemma \ref{lem1} together with \eqref{lem2:proof:eq1}, \eqref{lem2:proof:eq2} and \eqref{lem2:proof:eq4}, we obtain \eqref{eq-15} since
\begin{align*}
\partial_\lambda \Psi_{\mu}(0,s,\vx)&=\partial_\lambda\mathcal{F}_\mu^{s,1}(0,v_\mu,\vx)\\
&=\e\partial_x\Phi_{\mu'}(1,X_1(1))\partial_x\Phi_{\mu'}(1,X_2(1))\\
&=\e\partial_x\Phi_{\mu'}(1,X_1(1))\cdot \e\partial_x\Phi_{\mu'}(1,X_2(1))\\
&=\partial_x\Phi_{\mu'}(s,x_1)\partial_x\Phi_{\mu'}(s,x_2).
\end{align*}
\end{proof}

\begin{proof}[\bf Proof of Proposition \ref{prop0}] 
	Assume $h\neq 0.$ This proof has three major steps:
	
    \noindent{\bf Step I.}
	Define
	$$
	f(q)=\e\partial_x\Phi_{\mu_P}(\eta,h+z_1(q))\partial_x\Phi_{\mu_P}(\eta,h+z_2(q))
	$$
	for $q\in[-\eta,\eta]$, where $z_1(q)$ and $z_2(q)$ are jointly Gaussian with mean zero and covariance $\e z_1(q)^2=\xi'(\eta)=\e z_2(q)^2$ and $\e z_1(q)z_2(q)=\xi_0'(q).$
	We claim that $f$ maps $[-\eta,\eta]$ into itself and has a unique fixed point $q^*\in (0,\eta]$. Moreover, $q^*=\eta$ if $\xi=\xi_0$ and $q^*<\eta$ if $\xi\neq \xi_0$. To see these, recall from \eqref{thm1:eq1} and \eqref{thm1:eq2},
	\begin{align}
	\begin{split}
	\label{eq6}
	\e\partial_x\Phi_{\mu_P}\left(\eta,h+z_1(\eta)\right)^2&=\e\partial_x\Phi_{\mu_P}\left(\eta,h+z_2(\eta)\right)^2\\
	&=\e \partial_x\Phi_{\mu_P}\left(\eta,h+\int_0^{\eta}\zeta(r)^{1/2}dB(r)\right)^2=\eta
	\end{split}
	\end{align}
	and
	\begin{align}
	\begin{split}
	\label{eq7}
	\xi''(\eta)\e\partial_{xx}\Phi_{\mu_P}\left(\eta,h+z_1(\eta)\right)^2&=\xi''(\eta)\e\partial_{xx}\Phi_{\mu_P}\left(\eta,h+z_2(\eta)\right)^2\\
	&=
	\xi''(\eta)\e \partial_{xx}\Phi_{\mu_P}\left(\eta,h+\int_0^{\eta}\zeta(r)^{1/2}dB(r)\right)^2\leq 1.
	\end{split}
	\end{align}
	Using \eqref{eq6} and the Cauchy-Schwarz inequality, $f$ evidently maps $[-\eta,\eta]$ into itself, which implies the existence of a fixed point, say $q^*.$ To see its uniqueness, we apply the Gaussian integration by parts to obtain
	\begin{align*}
	f'(q)&=\xi_0''(q)\e \partial_{xx}\Phi_{\mu_P}(\eta,h+z_1(q))\partial_{xx}\Phi_{\mu_P}(\eta,h+z_2(q)).
	\end{align*}
	Since $\xi_0''(q)\leq \xi''(|q|)<\xi''(\eta)$ for $q\in(-\eta,\eta),$ using the Cauchy-Schwarz inequality and \eqref{eq7} to this formula leads to $f'<1$ on $(-\eta,\eta).$ So the fixed point $q^*$ is unique. Now since $\partial_x\Phi_{\mu_P}$ is odd and strictly increasing (see Lemma \ref{sec2:lem2}) and $h\neq 0$, one sees that $$
	f(0)=\bigl(\e\partial_x\Phi_{\mu_P}(\eta,h+z)\bigr)^2>0,
	$$
	where $z$ is Gaussian with mean zero and variance $\xi'(\eta).$ So $q^*\in (0,\eta].$ If $\xi=\xi_0$, \eqref{eq6} implies $q^*=\eta$; if $\xi\neq \xi_0$, then the Cauchy-Schwarz inequality and \eqref{eq6} leads to $q^*<\eta.$ This ends the proof of our claim.	

\noindent{\bf Step  II.} We check that
\begin{align}
\begin{split}\label{lem4:proof:eq2}
\Lambda_{\mu_P}(0,q)&=2\mathcal{P}(\mu_P),
\end{split}\\
\begin{split}\label{lem4:proof:eq3}
\partial_\lambda\Lambda_{\mu_P}(0,q)&=f(q)-q
\end{split}
\end{align}
for $|q|\leq \eta.$
Consider the variational representation \eqref{thm2:eq2} for $\Psi_{\mu_P}(\lambda,0,h,h)$ with $(s,t)=(0,\eta).$ Since $\alpha_{\mu_P}=0$ on $[0,\eta),$ 
\begin{align*}
\mathcal{F}_{\mu_P}^{0,\eta}(\lambda,v,h,h)&=\e\Psi_{\mu_P}\Bigl(\lambda, \eta,(h,h)+\int_0^{\eta}T(r)^{1/2}d\mathcal{B}(r)\Bigr),\,\,\forall v\in\mathcal{D}[0,\eta].
\end{align*}
Observe that from \eqref{eq-12}, $\int_0^{\eta}T(r)^{1/2}d\mathcal{B}(r)$ has the covariance structure
\begin{align*}
\int_0^{\eta}T(r)dr&=\int_0^{|q|}dr\left[\begin{array}{cc}
\xi''(r)&\iota\xi_0''(\iota r)\\
\iota\xi_0''(\iota r)&\xi''(r)
\end{array}\right]+\int_{|q|}^{\eta}dr
\left[\begin{array}{cc}
\xi''(r)&0\\
0&\xi''(r)
\end{array}\right]=\left[
\begin{array}{cc}
\xi'(\eta)&\xi'_0(q)\\
\xi_0'(q)&\xi'(\eta)
\end{array}\right].
\end{align*}
So we may as well write
\begin{align*}
\Psi_{\mu_P}(\lambda,0,h,h)&=\max_{v\in \mathcal{D}[0,\eta]}\mathcal{F}_{\mu_P}^{0,\eta}(\lambda,v,h,h)=\e \Psi_{\mu_P}\left(\lambda, \eta,h+z_1(q),h+z_2(q)\right),
\end{align*}
where $(z_1(q),z_2(q))$ is the Gaussian vector defined in  Step I.
Therefore, using \eqref{eq-14},
\begin{align*}
\Psi_{\mu_P}(0,0,h,h)&=\e \Psi_{\mu_P}\left(0, \eta,h+z_1(q),h+z_2(q)\right)\\
&=\e\Phi_{\mu_P}(\eta,h+z_1(q))+\e\Phi_{\eta_P}(\eta,h+z_2(q))\\
&=2\e\Phi_{\mu_P}(\eta,h+z)\\
&=2\Phi_{\mu_P}(0,h),
\end{align*}
where $z$ is a Gaussian random variable with mean zero and variance $\e z^2=\xi'(\eta)^2$ and the last equality used the assumption that $\alpha_{\mu}=0$ on $[0,|q|)$ and the variational representation \eqref{thm-2:eq1} for $\Phi_\mu(0,h)$ with $(s,t)=(0,\eta)$.
In addition, applying \eqref{eq-15},
\begin{align*}
\partial_\lambda \Psi_{\mu_P}(0,0,h,h)&=\e \partial_\lambda\Psi_{\mu_P}\left(0, \eta,h+z_1(q),h+z_2(q)\right)\\
&=\e\partial_x\Phi_{\mu_P}(0,h+z_1(q))\partial_x\Phi_{\mu_P}(0,h+z_2(q))\\
&=f(q).
\end{align*}
Using again $\alpha_{\mu_P}=0$ on $[0,|q|),$ we then obtain
\begin{align*}
\Lambda_{\mu_P}(0,q)&=2\log 2+\Psi_{\mu_P}(0,0,h,h)-\int_0^1\alpha_{\mu_P}(s)s\xi''(s)ds=2\mathcal{P}(\mu_P),\\
\partial_\lambda\Lambda_{\mu_P}(0,q)&=\partial_\lambda \Psi_{\mu_P}(0,0,h,h)-q=f(q)-q,
\end{align*}
which complete the verification of \eqref{lem4:proof:eq2} and \eqref{lem4:proof:eq3}. 

\noindent{\bf Step III.} From \eqref{lem4:proof:eq3} and Step I, we know $\partial_\lambda\Lambda_{\mu_P}(0,q)\neq 0$ for any $q\in [-\eta,\eta]\setminus\{q^*\}.$ Depending on the sign of this quantity, we may decrease or increase $\lambda$ slightly to obtain $\Lambda_{\mu_P}(\lambda,q)<\Lambda_{\mu_P}(0,q).$ As a result, $\Lambda(q)<2\mathcal{P}(\mu_P)$ by the definition of $\Lambda(q)$ and \eqref{lem4:proof:eq2}. This finishes our proof.

\end{proof}

\subsection{Behavior of $\Lambda$ outside of $[-\eta,\eta]$}

For notational convenience, we set $\zeta(s)=\xi''(s)$ and $\zeta_0(s)=\xi_0''(\iota s)$ for $s\in[0,1].$ Note that since $\xi$ and $\xi_0$ are convex and are not identically equal to zero, the function $\zeta$ is positive on $(0,1]$ and so is $\zeta_0$ if $\iota=1.$ In addition, $\zeta_0\geq 0$ on $(0,1]$ and $\zeta_0=0$ for at most a finite number of points if $\iota=-1.$ The following proposition takes care of the behavior of $\Lambda$ on $[-1,-\eta)\cup(\eta,1].$ 

\begin{proposition} 
\label{prop2} The following two statements hold.
\begin{itemize}
\item[$(i)$] For $-1\leq q<-\eta,$ if $\xi=\xi_0$ is even and $h\neq 0,$ then $\Lambda(q)<2\mathcal{P}(\mu_P)$.
\item[$(ii)$] For $|q|>\eta,$ if $\zeta_0<\zeta$ and $
\zeta/(\zeta+\zeta_0)
$
is nondecreasing on $(0,1],$ then $\Lambda(q)<2\mathcal{P}(\mu_P)$.
\end{itemize}
\end{proposition}

The essential idea to prove this proposition is to construct relevant $\mu\in\mathcal{M}$ depending on $q$ and $\mu_P$ such that
\begin{align}
\label{sec4.2:eq1}
\int_0^1\alpha_\mu(s)s\xi''(s)ds+\int_0^{|q|}\alpha_\mu(s)s\xi_0''(\iota s)ds=\int_0^1\alpha_{\mu_P}(s)s\xi''(s)ds
\end{align}
and
\begin{align}
\label{sec4.2:eq2}
\Psi_{\mu}(0,0,h,h)<2\Phi_{\mu_P}(0,h).
\end{align}
Once these are established, it will follow by definition that $\Lambda(q)\leq\Lambda_{\mu}(0,q)<2\mathcal{P}(\mu_P).$ In order to get \eqref{sec4.2:eq1}, one natural choice of $\mu$ is via \eqref{lem6:eq1} below. The major obstacle here comes from the derivation of \eqref{sec4.2:eq2} for such a choice of $\mu$. This will be handled through the variational representation for $\Psi_{\mu}$ and $\Phi_{\mu_P}$. A key lemma we will need along the line is the global uniqueness of the maximizer for $\Phi_{\mu_P}$.

\begin{lemma}\label{lem00}
Let $0\leq s<t\leq 1$. Suppose that $u^*$ is a maximizer of $\Phi_\mu(0,x)=\max_{u\in D[0,t]}F_\mu^{0,t}(u,x)$. If $\alpha_\mu>0$ on $(s,t),$ then $u^*(r)=\partial_x\Phi_\mu(r,X(r))$ for $s\leq r\leq t$, where 
$$
X(r)=x+\int_0^r\alpha_\mu(w) \zeta(w) \partial_x\Phi_\mu(w,X(w))dw+\int_0^r\zeta(w)^{1/2}dB(w),\,\,\forall s\leq r\leq t.
$$
In other words, the maximizer is unique under the assumption $\alpha_\mu>0$ on $(s,t).$ 
\end{lemma}

\begin{proof}
Let $\{a_i\}_{i=0}^n$ be a regular partition of $[s,t]$ with $\int_{a_{i}}^{a_{i+1}}\alpha_\mu(r)\zeta(r)dr<1$ for $1\leq i<n.$ Define $u_i\in D[a_{i},t]$ by $u_i(w)=u^*(w)$ for $a_{i}\leq w\leq t$ and $v_i\in D[0,a_{i}]$ by $v_i(w)=u^*(w)$ for $0\leq w\leq a_i$. Set
\begin{align*}
&y_i=x+\int_{0}^{a_i}\alpha_\mu(w)\zeta(w)v_i(w)dw+\int_0^{a_i}\zeta(w)^{1/2}dB(w)
\end{align*}
Using conditional expectation,
\begin{align*}
\Phi_\mu(0,x)&=\e\bigl(\e\bigl[C_{\mu}^{0,t}(u^*,x)-L_{\mu}^{0,t}(u^*)\big|y_i\bigr]\bigr)\\
&=\e\bigl(\e\bigl[C_{\mu}^{a_i,t}(u_i,y_i)-L_{\mu}^{a_i,t}(u_i)\big|y_{i}\bigr]\bigr)-\e L_\mu^{0,a_i}(v_i)\\
&\leq \e\Phi_\mu(a_i,y_i)-\e L_\mu^{0,a_i}(v_i)\\
&=F_\mu^{0,a_i}(v_i,x)\\
&\leq \Phi_\mu(0,x),
\end{align*}
which implies that $F_{\mu}^{0,a_i}(v_i,x)$'s are the same for all $0\leq i\leq n.$ Using this, we obtain that
\begin{align*}
\e C_\mu^{0,a_{i}}(v_{i},x)-\e L_\mu^{0,a_{i}}(v_{i})
&=F_{\mu}^{0,a_{i}}(v_{i},x)\\
&=F_{\mu}^{0,a_{i+1}}(v_{{i+1}},x)\\
&=\e C_\mu^{0,a_{i+1}}(v_{i+1},x)-\e L_\mu^{0,a_{i+1}}(v_{i+1}).
\end{align*}
and thus,
\begin{align}\label{eq00}
\e C_\mu^{0,a_{i}}(v_{i},x)
&=\e C_\mu^{0,a_{i+1}}(v_{i+1},x)-\e L_\mu^{a_i,a_{i+1}}(u_{i}')\notag\\
&=\e C_\mu^{a_{i},a_{i+1}}(u_{i}',y_{i})-\e L_\mu^{a_{i},a_{i+1}}(u_i')\notag\\
&=\e \bigl(\e \bigl[C_\mu^{a_{i},a_{i+1}}(u_{i}',y_i)-\e L_\mu^{a_{i},a_{i+1}}(u_i')\big|y_i\bigr]\Bigr)\\
&\leq \e \max_{u'\in D[a_{i},a_{i+1}]}F_\mu^{a_{i},a_{i+1}}(u',y_{i})\notag\\
&=\e \Phi_\mu(a_{i},y_{i}),\notag
\end{align}
where $u_i'\in D[a_{i},a_{i+1}]$ is the restriction of $u^*$ to $[a_i,a_{i+1}].$ Since 
\begin{align*}
\max_{u'\in D[a_{i},a_{i+1}]}F_\mu^{a_{i},a_{i+1}}(u',y)=\Phi_\mu(a_{i},y),\,\,\forall y\in\mathbb{R}
\end{align*}
and
$
\e C_\mu^{0,a_{i}}(v_i,x)=\e\Phi_{\mu}(a_{i},y_i),
$
these and \eqref{eq00} force that conditioning on $y_i,$ $u_i'$ is the maximizer to the variational problem $\max_{u'\in D[a_{i},a_{i+1}]}F_\mu^{a_{i},a_{i+1}}(u',y_{i})$. Therefore, applying the local uniqueness of the maximizer for $(s,t)=(a_i,a_{i+1})$ in Theorem \ref{thm-2} leads to $u_i'(r)=\partial_x\Phi_\mu(r,X_i(r))$ on $[a_{i},a_{i+1}],$ where $X_i=(X_i(w))_{a_{i}\leq w\leq{a_{i+1}}}$ is the solution to 
\begin{align*}
X_i(r)&=y_{i}+\int_{a_{i}}^{a_{i+1}}\alpha_\mu(w)\zeta(w) \partial_x\Phi_\mu(w,X_i(w))dw+\int_{a_{i}}^{a_{i+1}}\zeta(w)^{1/2}dB(w).
\end{align*}
Concatenating all these from $i=0$ to $n-1$  together gives the announced result.
\end{proof}

The proposition below is the core ingredient of the matter that gives a quantitative error estimate between the one and two dimensional Parisi PDEs for a specific choice of $\mu.$ 

\begin{proposition}\label{lem6}
Assume that $|q|>\eta$ and
\begin{align}\label{lem6:eq0}
\frac{\zeta(s)}{\zeta(s)+\zeta_0(s)}
\end{align}
is nondecreasing on $(0,1]$. Define $\mu\in\mathcal{M}$ by
\begin{align}\label{lem6:eq1}
\alpha_\mu(s)&=\left\{
\begin{array}{ll}
\frac{\alpha_{\mu_P}(s)\zeta(s)}{\zeta(s)+\zeta_0(s)},&\mbox{if $s\in [0,|q|)$},\\
\\
\alpha_{\mu_P}(s),&\mbox{if $s\in[|q|,1]$}.
\end{array}
\right.
\end{align}
\begin{enumerate}
\item[$(i)$] We have that
\begin{align}
\begin{split}
\label{eq4}
\Psi_\mu(0,0,\vx)&\leq \Phi_{\mu_P}(0,x_1)+\Phi_{\mu_P}(0,x_2)
-\frac{1}{2}\int_{0}^{|q|}\frac{\alpha_{\mu_P}\zeta\zeta_0(\zeta-\zeta_0)}{(\zeta+\zeta_0)^2}\e\left(v_1-\iota v_2\right)^2dw,
\end{split}
\end{align}
where $v_\mu=(v_1,v_2)$ is the maximizer to the variational problem \eqref{thm2} for $\Psi_\mu(0,0,x_1,x_2)$ using $(s,t)=(0,|q|).$ 

\item[$(ii)$] Define 
\begin{align}
\begin{split}\label{eq01}
\left(u_1(r),u_2(r)\right)&=\frac{1}{\zeta(r)+\zeta_0(r)}T(r)v_\mu(r),
\end{split}\\
\begin{split}
\label{eq001}
(B_1(r),{B}_2(r))&=\frac{1}{\zeta(r)^{1/2}}T(r)^{1/2}\mathcal{B}(r)
\end{split}
\end{align}
for $0\leq r\leq |q|.$ If $$
\Psi_\mu(0,0,\vx)=\Phi_{\mu_P}(0,x_1)+\Phi_{\mu_P}(0,x_2),$$
then $u_1$ and $u_2$ are the maximizers for the variational problem \eqref{thm-2:eq1} of $\Phi_{\mu_P}(0,x_1)$ and $\Phi_{\mu_P}(0,x_2)$ using $(s,t)=(0,|q|)$ with respect to the standard Brownian motions $B_1$ and $B_2$, respectively. Moreover, on the interval $[\eta,|q|]$, they are equal to 
\begin{align*}
u_1(r)&=\partial_x\Phi_{\mu_P}(r,X_1(r)),\\
u_2(r)&=\partial_x\Phi_{\mu_P}(r,X_2(r)),
\end{align*}
where $(X_1(r))_{0\leq r\leq|q|}$ and $(X_2(r))_{0\leq r\leq |q|}$ satisfy
\begin{align*}
X_1(r)&=x_1+\int_{0}^r\alpha_{\mu_P}(w)\zeta(w)\partial_x\Phi_{\mu_P}(w,X_1(w))dw+\int_{0}^r\zeta(w)^{1/2}dB_1(w),\\
X_2(r)&=x_2+\int_{0}^r\alpha_{\mu_P}(w)\zeta(w)\partial_x\Phi_{\mu_P}(w,X_2(w))dw+\int_{0}^r\zeta(w)^{1/2}dB_2(w).
\end{align*}
\end{enumerate}
\end{proposition}

\begin{proof}
Note that the well-definedness of $\mu$ is guaranteed by \eqref{lem6:eq0}. Let $v_\mu=(v_1,v_2)$ be the maximizer to the variational problem \eqref{thm2} for $\Psi_{\mu}$ with $(s,t)=(0,|q|).$ Set $(u_1,u_2)$ via \eqref{eq01}. Here $u_1,u_2$ are progressively measurable processes with respect to the filtration $(\mathscr{G}_r)_{r\geq 0}$ and ${B}_1,{B}_2$ are (correlated) standard Brownian motions. Denote by $$
\left(C_{\mu_P,1}^{0,|q|},L_{\mu_P,1}^{0,|q|},F_{\mu_P,1}^{0,|q|}\right)\,\,\mbox{and}\,\,\left(C_{\mu_P,2}^{0,|q|},L_{\mu_P,2}^{0,|q|},F_{\mu_P,2}^{0,|q|}\right)$$
the functionals defined in the same away as \eqref{eq2} by using ${B}_1$ and ${B}_2$, respectively. Observe that from \eqref{eq-14} and the definition of $u_1,u_2,$
\begin{align*}
\mathcal{C}_\mu^{0,|q|}(0,v_\mu,\vx)&=\Phi_{\mu_P}\Bigl(|q|,x_1+\int_{0}^{|q|}\alpha_{\mu_P}(w)\zeta(w)u_1(w)dw+\int_{0}^{|q|}\zeta(w)^{1/2}d{B}_1(w)\Bigr)\\
&+\Phi_{\mu_P}\Bigl(|q|,x_2+\int_{0}^{{|q|}}\alpha_{\mu_P}(w)\zeta(w)u_2(w)dw+\int_{0}^{|q|}\zeta(w)^{1/2}d{B}_2(w)\Bigr)\\
&=C_{\mu_P,1}^{0,|q|}(u_1,x_1)+C_{\mu_P,2}^{0,|q|}(u_2,x_2).
\end{align*}
In addition, noting that a direct computation gives
\begin{align*}
\left(u_1(r),u_2(r)\right)&=\left(\frac{\zeta(r)v_1(r)+\iota\zeta_0(r)v_2(r)}{\zeta(r)+\zeta_0(r)},\frac{\iota\zeta_0(r)v_1(r)+\zeta(r)v_2(r)}{\zeta(r)+\zeta_0(r)}\right),
\end{align*}
it follows
\begin{align*}
&\frac{1}{\zeta+\zeta_0}\left<Tv_\mu,v_\mu\right>-u_1^2-u_2^2\\
&=\left(\frac{\zeta}{\zeta+\zeta_0}-\frac{\zeta^2+\zeta_0^2}{(\zeta+\zeta_0)^2}\right)(v_1^2+v_2^2)+2\iota \zeta_0\left(\frac{1}{\zeta+\zeta_0}-\frac{2\zeta}{(\zeta+\zeta_0)^2}\right)v_1v_2\\
&=\frac{\zeta_0(\zeta-\zeta_0)}{(\zeta+\zeta_0)^2}\left(v_1-\iota v_2\right)^2,
\end{align*}
which implies
\begin{align*}
\mathcal{L}_{\mu}^{0,|q|}(v_\mu)&=L_{\mu_P,1}^{0,|q|}(u_1)+L_{\mu_P,2}^{0,|q|}(u_2)
+\frac{1}{2}\int_{0}^{|q|}\frac{\alpha_{\mu_P}\zeta\zeta_0(\zeta-\zeta_0)}{(\zeta+\zeta_0)^2}\e\left(v_1-\iota v_2\right)^2dw.
\end{align*}
Combining these together, the variational representations for $\Psi_{\mu}(0,0,\vx)$ and $\Phi_{\mu_P}(0,h)$ yield \eqref{eq4} since
\begin{align*}
\Psi_\mu(0,0,\vx)
&=\mathcal{F}_\mu^{0,|q|}(0,v_\mu,\vx)\\
&= F_{\mu_P,1}^{0,|q|}(u_{1},x_1)+F_{\mu_P,2}^{0,|q|}(u_2,x_2)-\frac{1}{2}\int_{0}^{|q|}\frac{\alpha_{\mu_P}\zeta\zeta_0(\zeta-\zeta_0)}{(\zeta+\zeta_0)^2}\e\left(v_1-\iota v_2\right)^2dw\\
&\leq \Phi_{\mu_P}(0,x_1)+\Phi_{\mu_P}(0,x_2)-\frac{1}{2}\int_{0}^{|q|}\frac{\alpha_{\mu_P}\zeta\zeta_0(\zeta-\zeta_0)}{(\zeta+\zeta_0)^2}\e\left(v_1-\iota v_2\right)^2dw.
\end{align*}
If $\Psi_\mu(0,0,\vx)=\Phi_{\mu_P}(0,x_1)+\Phi_{\mu_P}(0,x_2)$, this inequality implies that $u_1$ and $u_2$ are the maximizers of the variational representations, $$
\Phi_\mu(0,x_1)=\max_{u\in D[0,|q|]}F_{\mu_P,1}^{0,|q|}(u,x_1)\,\,\mbox{and}\,\,\Phi_\mu(0,x_2)=\max_{u\in D[0,|q|]}F_{\mu_P,2}^{0,|q|}(u,x_2)
$$ corresponding to the Brownian motions $B_1$ and $B_2$ respectively. Since $\alpha_\mu>0$ on $(\eta,|q|],$ Lemma \ref{lem00} concludes $(ii).$
\end{proof}

\begin{proof}[\bf Proof of Proposition \ref{prop2}] First note that the measure $\mu$ in \eqref{lem6:eq1} is well-defined since the function $\zeta/(\zeta+\zeta_0)$ under both assumptions $(i)$ and $(ii)$ is nondecreasing on $(0,1].$ We plug this $\mu$ into \eqref{eq-11} and let $\lambda=0$ to obtain $$
\Lambda(q)\leq 2\log 2+\Psi_\mu(0,0,h,h)-\int_0^1\alpha_{\mu_P}(s)s\xi''(s)ds.
$$
Thus, to finish the proof, we only need to verify that $\Psi_\mu(0,0,h,h)<2\Phi_{\mu_P}(0,h).$ Suppose the equality holds. Proposition \ref{lem6}$(ii)$ implies
that for any $\eta\leq r\leq |q|,$
\begin{align*}
u_1(r)&=\partial_x\Phi_{\mu_P}(r,X_1(r)),\\
u_2(r)&=\partial_x\Phi_{\mu_P}(r,X_2(r)),
\end{align*}
where $(X_1(r))_{0\leq r\leq|q|}$ and $(X_2(r))_{0\leq r\leq |q|}$ satisfy
\begin{align}
\begin{split}\label{eq5}
X_1(r)&=h+\int_{0}^r\alpha_{\mu_P}(w)\zeta(w)\partial_x\Phi_{\mu_P}(w,X_1(w))dw+\int_{0}^r\zeta(w)^{1/2}dB_1(w),\\
X_2(r)&=h+\int_{0}^r\alpha_{\mu_P}(w)\zeta(w)\partial_x\Phi_{\mu_P}(w,X_2(w))dw+\int_{0}^r\zeta(w)^{1/2}dB_2(w).
\end{split}
\end{align}
Our proof will clearly be completed by the following two cases.
\smallskip

\noindent{\bf Case I:} $-1\leq q<-\eta$, $\xi=\xi_0$ is even and $h\neq 0.$ Since $\iota=-1$, these assumptions combined with \eqref{eq01} and \eqref{eq001} lead to $u_1=-u_2$ and $B_1=-B_2$. Consequently, adding the two equations in \eqref{eq5} together implies $X_1(r)+X_2(r)=2h$ for $q\leq r\leq -\eta.$ On the other hand, since $\partial_x\Phi_{\mu_P}(r,\cdot)$ is odd and strictly increasing from Lemma \ref{sec2:lem2}, the equation $$
\partial_x\Phi_{\mu_P}(r,X_1(r))=u_1(r)=-u_2(r)=\partial_x\Phi_{\mu_P}(r,-X_2(r))
$$
deduces $X_1(r)=-X_2(r)$, which contradicts $X_1(r)+X_2(r)=2h$ since $h\neq 0.$ 

\smallskip

\noindent{\bf Case II:} $\zeta_0<\zeta$ on $(0,1]$. Since $\zeta>\zeta_0\geq 0$ and $\zeta_0=0$ for at most a finite number of points, we deduce from \eqref{eq4} and the continuity of $v_1,v_2$ that $v_1=\iota v_2$ on $[\eta,|q|]$. From \eqref{eq01}, it then follows that $u_1=\iota u_2$ on $[\eta,|q|].$ Again, using the facts that $\partial_x\Phi_{\mu_P}(r,\cdot)$ is odd and strictly increasing, we conclude $X_1=\iota X_2$ on $[\eta,|q|]$ and from \eqref{eq5}, for $r\in [\eta,|q|],$
\begin{align*}
0&=X_1(r)-\iota X_2(r)=(1-\iota)h+\int_0^r\zeta(w)^{1/2}d\bigl(B_1(w)-\iota B_2(w)\bigr).
\end{align*}
This forces that $B_1=\iota B_2$ and therefore, \eqref{eq001} implies that $\zeta(r)^2-\zeta_0(r)^2=\det T(r)=0$ for $r\in [\eta,|q|]$. This leads to a contradiction since $\zeta>\zeta_0\geq 0.$
\end{proof}

\subsection{Proof of Theorems \ref{pp}, \ref{npp} and \ref{dc}}

Before we start, it is crucial to notice that $\Psi_{\mu}(\lambda,0,h,h)$ is a continuous function in $q\in[-1,1]$ for any $\mu\in\mathcal{M}$ and $\lambda\in\mathbb{R}$, which can be easily shown by following a similar argument as in the proof of Theorem \ref{sec1.2:thm1}. Thus, $\Lambda$ is upper semicontinuous on $[-1,1].$

\begin{proof}[\bf Proof of Theorem \ref{pp}]
Note that $H_N^1=H_N=H_N^2$ since $\xi=\xi_0.$ Let $\varepsilon>0.$ From the upper semicontinuity of $\Lambda$ on $[-1,\eta-\varepsilon]$, we denote by $q'$ the maximizer of $$
\max_{q\in[-1,\eta-\varepsilon]}\Lambda(q).
$$
If the assumption $(i)$ holds, then Proposition \ref{prop0} and the first assertion of Proposition \ref{prop2} together implies $\Lambda(q)<2\mathcal{P}(\mu_P)$ for $q\in[-1,\eta-\varepsilon]$ and thus, 
\begin{align}
\label{proof:thmpp:eq1}
\Lambda(q)\leq \Lambda(q')<2\mathcal{P}(\mu_P),\,\,\forall q\in[-1,\eta-\varepsilon].
\end{align} 
Now suppose that the condition $(ii)$ is true. Then the series $\xi$ must contain some term $\beta_{p}^2s^p$ with $\beta_p\neq 0$ for some odd $p$. This implies that for any $q\in[-1,-\eta),$ 
\begin{align}
\label{eq}
\zeta_0(s)=\xi''(-s)<\xi''(s)=\zeta(s),\,\,\forall s\in (0,1],
\end{align}
which combined with \eqref{pp:eq2} yields $\Lambda(q)<2\mathcal{P}(\mu_P)$ for $q\in [-1,-\eta)$ by the second assertion of Proposition \ref{prop2}. Since $h\neq 0$, we can use Proposition \ref{prop0} to obtain $\Lambda(q)<2\mathcal{P}(\mu_P)$ for $q\in [-\eta,\eta-\varepsilon]$ and consequently \eqref{proof:thmpp:eq1} is valid. In summary, the two assumptions $(i)$ and $(ii)$ lead to
\begin{align*}
\limsup_{N\rightarrow\infty}\max_{q\in S_N\cap[-1,\eta-\varepsilon]}\frac{1}{N}\e\log\sum_{R_{1,2}=q}\exp\bigl(H_N(\vsi^1)+H_N(\vsi^2)\bigr)<2\mathcal{P}(\mu_P).
\end{align*}
Finally, from this inequality, \eqref{eq-5} can be obtained by using the Gaussian concentration of measure and the Parisi formula. Since this part of the argument is very standard and has appeared in several places, e.g. \cite[Section 14.12]{Tal11}, we omit the details.
\end{proof}

\begin{proof}[\bf Proof of Theorem \ref{npp}]
Again $H_N^1=H_N=H_N^2$. Note that $\eta=0$ since $h=0.$ Recall the maximizer $q'$ from the proof of Theorem \ref{pp}. From the given assumption of Theorem \ref{npp}, one sees that \eqref{eq} is also valid. Thus, the second assertion of Proposition  \ref{prop2} implies $\Lambda(q)<2\mathcal{P}(\mu_p)$ for all $q\in[-1,-\varepsilon]$ and as a result,
$$
\Lambda(q)\leq \Lambda(q')<2\mathcal{P}(\mu_p),\,\,\forall q\in[-1,-\varepsilon],
$$ 
from which it follows that
\begin{align*}
\limsup_{N\rightarrow\infty}\max_{q\in S_N\cap[-1,-\varepsilon]}\frac{1}{N}\e\log\sum_{R_{1,2}=q}\exp\bigl(H_N(\vsi^1)+H_N(\vsi^2)\bigr)<2\mathcal{P}(\mu_P).
\end{align*}
The rest of the proof can be completed by an identical argument as the last part of the proof of Theorem \ref{pp}.
\end{proof}

\begin{proof}[\bf Proof of Theorem \ref{dc}] Note that $\xi\neq \xi_0.$ Let $q^*\in (0,\eta)$ be the constant stated in Proposition \ref{prop0} if $h\neq 0$ and set $q^*=0$ if $h=0.$
From the upper semicontinuity of $\Lambda,$ for $\varepsilon>0,$ let $q''$ be the maximizer of 
$$
\max_{q\in[-1,1]:|q-q^*|\geq \varepsilon}\Lambda(q).
$$ 
Note that from the assumptions \eqref{dc:eq2} and \eqref{dc:eq3}, 
\begin{align}
\label{proof:dc:eq1}
\Lambda(q)<2\mathcal{P}(\mu_P),\,\,\forall q\in[-1,-\eta)\cup(\eta,1]
\end{align}
by the second statement of Proposition \ref{prop2}.
If $h=0$, then $\eta=q^*=0$ and this inequality implies
\begin{align}\label{proof:dc:eq2}
\Lambda(q)\leq \Lambda(q'')<2\mathcal{P}(\mu_P),\,\,\forall q\in[-1,1]\,\,\mbox{with}\,\, |q-q^*|\geq\varepsilon.
\end{align}
If $h\neq 0,$ then Proposition \ref{prop0} gives $\Lambda(q)<2\mathcal{P}(\mu_P)$ for $q\in[-\eta,\eta]\setminus\{q^*\}.$ This together with \eqref{proof:dc:eq1} concludes \eqref{proof:dc:eq2}
by the second assertion of Proposition \ref{prop2}. Therefore, we have shown that
\begin{align*}
\limsup_{N\rightarrow\infty}\max_{q\in S_N:|q-q^*|\geq\varepsilon}\frac{1}{N}\e\log\sum_{R_{1,2}=q}\exp\bigl(H_N^1(\vsi^1)+H_N^2(\vsi^2)\bigr)<2\mathcal{P}(\mu_P).
\end{align*}
Using this inequality, \eqref{chaos} follows by applying the Gaussian concentration of measure and the Parisi formula. Once again, we skip this part of the argument as it can be found in great detail in the proof of \cite[Theorem 7]{Chen14}.
\end{proof}

\begin{center}
{\bf \large Appendix}
\end{center}

\begin{proof}[\bf Proof of Lemma \ref{lem3}] 
We argue by applying the Gaussian integration by parts formula.
Define
\begin{align*}
c_1(s)&=\iota(\xi_{1,2}'(\iota\rho_{1,2}(b))-\xi_{1,2}'(\iota\rho_{1,2}(s)))=\int_s^b\rho_{1,2}'(l)\xi_{1,2}''(\iota\rho_{1,2}(l))dl\geq 0,\\
c_2(s)&=\iota(\xi_{2,1}'(\iota\rho_{2,1}(b))-\xi_{2,1}'(\iota\rho_{2,1}(s)))=\int_s^b\rho_{2,1}'(l)\xi_{2,1}''(\iota\rho_{2,1}(l))dl\geq 0,\\
d_1(s)&=\xi_{1,1}'(\rho_{1,1}(b))-\xi_{1,1}'(\rho_{1,1}(s))-c_1(s),\\
d_2(s)&=\xi_{2,2}'(\rho_{2,2}(b))-\xi_{2,2}'(\rho_{2,2}(s))-c_2(s).
\end{align*}
We parametrize $(y_1(s),y_2(s))$ as 
$$
(y_1(s),y_2(s))=\left(\iota\sqrt{c_1(s)}z_0+\sqrt{d_1(s)}z_1,\sqrt{c_2(s)}z_0+\sqrt{d_2(s)}z_2\right),
$$ 
where $z_0,z_1,z_2$ are i.i.d. standard Gaussian. 
Note $c_1=c_2$ by the symmetry of $T$. Recall $\zeta_{\ell,\ell'}$ from \eqref{eq-10}. Observe that
\begin{align*}
\e y_1'(s)y_1(s)&=-\frac{\zeta_{1,1}(s)}{2},\\
\e y_2'(s)y_2(s)&=-\frac{\zeta_{2,2}(s)}{2},\\
\e y_1'(s)y_2(s)&=-\frac{\zeta_{1,2}(s)}{2}=-\frac{\zeta_{2,1}(s)}{2}=\e y_2'(s)y_1(s),
\end{align*} 
where $y_1'$ and $y_2'$ are the derivatives of $y_1$ and $y_2$ with respect to $s$ respectively.
From the growth condition of $A$, it allows us to apply the Gaussian integration by parts to obtain
\begin{align*}
\partial_sL
&=\frac{1}{\e e^{mA}}\e\left[ y_1'\partial_{x_1}A+y_2'\partial_{x_2}A\right]e^{mA}\\
&=\frac{1}{\e e^{mA}}\Bigl(\e (y_1'y_1)\e (\partial_{x_1x_1}A+m(\partial_{x_1}A)^2)e^{mA}+\e (y_1'y_2)\e (\partial_{x_1x_2}A+m\partial_{x_1}A\partial_{x_2}A)e^{mA}\Bigr)\\
&\quad+\frac{1}{\e e^{mA}}\Bigl(\e (y_2'y_2)\e (\partial_{x_2x_2}A+m(\partial_{x_2}A)^2)e^{mA}+\e (y_2'y_1)\e (\partial_{x_2x_1}A+m\partial_{x_1}A\partial_{x_2}A)e^{mA}\Bigr)\\
&=-\frac{1}{2\e e^{mA}}\e\left[\zeta_{1,1}\left(\partial_{x_1x_1}A+m(\partial_{x_1}A)^2\right)+\zeta_{1,2}\left(\partial_{x_1x_2}A+ m(\partial_{x_1}A)(\partial_{x_2}A)\right)\right]e^{mA}\\
&\quad-\frac{1}{2\e e^{mA}}\e\left[\zeta_{2,2}\left(\partial_{x_2x_2}A+m(\partial_{x_2}A)^2\right)+\zeta_{2,1}\left(\partial_{x_1x_2}A+ m(\partial_{x_1}A)(\partial_{x_2}A)\right)\right]e^{mA}\\
&=-\frac{1}{2\e e^{mA}}\e\left[\left<T ,\triangledown^2A\right>+m\left<T\triangledown A,\triangledown A\right>\right]e^{mA}.
\end{align*}
On the other hand, a direct computation gives
\begin{align}
\begin{split}\label{app:eq1}
\partial_{x_1}L=\frac{\e \partial_{x_1}Ae^{mA}}{\e e^{mA}},\\
\partial_{x_2}L=\frac{\e \partial_{x_2}Ae^{mA}}{\e e^{mA} }
\end{split}
\end{align}
and
\begin{align*}
\partial_{x_1x_1}L&=\frac{\e(\partial_{x_1x_1}A+m(\partial_{x_1}A)^2)e^{mA}}{\e e^{mA}}-m\left(\frac{\e \partial_{x_1}A e^{mA}}{\e e^{mA}}\right)^2,\\
\partial_{x_2x_2}L&=\frac{\e(\partial_{x_2x_2}A+m(\partial_{x_2}A)^2)e^{mA}}{\e e^{mA}}-m\left(\frac{\e \partial_{x_2}e^{mA}}{\e e^{mA}}\right)^2,\\
\partial_{x_1x_2}L&=\partial_{x_2x_1}L=\frac{\e(\partial_{x_1x_2}A+m(\partial_{x_1}A)(\partial_{x_2}A))e^{mA}}{\e e^{mA}}
-m\left(\frac{\e \partial_{x_1} e^{mA}}{\e e^{mA}}\right)\left(\frac{\e \partial_{x_2}A\exp mA}{\e e^{mA}}\right).
\end{align*}
Using these, one may easily check that 
\begin{align*}
&\left<T,\triangledown^2L\right>+m\left<T\triangledown L,\triangledown L\right>\\
&=-\frac{1}{e^{mA}}\e\left[\left<T,\triangledown^2A\right>+m\left<T \triangledown A,\triangledown A\right>\right]e^{mA}\\
&=-2\partial_sL,
\end{align*}
which gives \eqref{lem3:eq1}. If $\partial_{x_i}A$ is uniformly bounded by $1$, then \eqref{app:eq1} clearly yields $|\partial_{x_i}L|\leq 1.$
\end{proof}

\end{document}